\documentclass[12 pt]{article}

\usepackage[
    margin=2.2cm,
    top=2cm,
    bottom=2cm,
    includefoot,
    footskip=20pt,
]{geometry}
\usepackage[utf8]{inputenc}
\usepackage[T1]{fontenc}
\usepackage[dvipsnames]{xcolor}
\usepackage[normalem]{ulem}
\usepackage{etoolbox}

\usepackage{comment}
\usepackage{quiver}
\usepackage{cancel}
\usepackage{setspace}

\usepackage{hyperref}
\hypersetup{hypertexnames = false, bookmarksdepth = 2, bookmarksopen = true, colorlinks, linkcolor = blue, citecolor = blue, urlcolor = blue, pdfstartview={XYZ null null 1}}

\usepackage{
    amsfonts,amsmath,amsthm,amssymb,mathtools,thmtools,
    tikz-cd,color,
    xparse,xspace,
    enumitem,
    microtype,lmodern
}

\usepackage[capitalise]{cleveref}

\theoremstyle{plain}
\newtheorem{theorem}{Theorem}[section]

\newtheorem{corollary}[theorem]{Corollary}

\newtheorem{lemma}[theorem]{Lemma}
\newtheorem{proposition}[theorem]{Proposition}

\newtheorem{conjecture}[theorem]{Conjecture}

\newtheorem{assumption}[theorem]{Assumption}
\theoremstyle{definition}
\newtheorem{definition}[theorem]{Definition}
\newtheorem{notation}[theorem]{Notation}
\newtheorem{remark}[theorem]{Remark}
\newtheorem{construction}[theorem]{Construction}
\newtheorem{setup}[theorem]{Setup}
\newtheorem{example}[theorem]{Example}

\declaretheoremstyle[
  spaceabove = 3pt,
  spacebelow = 3pt,
  bodyfont=\normalfont\itshape,
]{alpha}

\makeatletter
\renewcommand{\paragraph}{%
  \@startsection{paragraph}{4}%
  {\z@}{1.5ex \@plus 1ex \@minus .2ex}{-1em}%
  {\normalfont\normalsize\bfseries}%
}
\makeatother

\DeclareMathOperator\cc{c}
\DeclareMathOperator\ch{ch}

\DeclareMathOperator\Exc{Exc}

\DeclareMathOperator\NE{NE}
\DeclareMathOperator\PC{PC}

\DeclareMathOperator\PR{PR}

\DeclareMathOperator\RPC{RPC}
\DeclareMathOperator\RPR{RPR}

\DeclareMathOperator\codim{codim}

\newcommand\CC{\ensuremath{\mathbb{C}}}

\newcommand\PP{\ensuremath{\mathbb{P}}}

\newcommand\ZZ{\ensuremath{\mathbb{Z}}}

\newcommand\cE{\mathcal{E}}

\newcommand\cL{\mathcal{L}}

\newcommand\cO{\mathcal{O}}

\newcommand\cU{\mathcal{U}}

\mathchardef\mhyphen="2D

\newcommand{\minRC}{m}
\newcommand{\cpc}{S_x}

\usepackage[utf8]{inputenc}
\usepackage[english]{babel}
\usepackage[backend = biber,
    style = alphabetic,
    citestyle = alphabetic,
    maxbibnames=10,
    giveninits=true,
    doi=false,
    url=false
]{biblatex}
\AtBeginBibliography{\footnotesize}

\usepackage{csquotes}
\DeclareFieldFormat{postnote}{#1}

\newcommand{\Addresses}{{% additional braces for segregating \footnotesize
  \bigskip
  \footnotesize

C.~Araujo, \textsc{IMPA, Estrada Dona Castorina 110, 22460-320 Rio de Janeiro, Brazil}\par\nopagebreak
\textit{E-mail address}: \texttt{caraujo@impa.br}

\medskip

R.~Beheshti, \textsc{Department of Mathematics, Washington University in St. Louis, MO 63130}\par\nopagebreak
\textit{E-mail address}: \texttt{beheshti@wustl.edu}

\medskip

A.-M.~Castravet, \textsc{Universit\'e Paris-Saclay, UVSQ, CNRS, Laboratoire de Math\'ematiques de Versailles, 78000, Versailles, France}\par\nopagebreak
\textit{E-mail address}: 
\texttt{ana-maria.castravet@uvsq.fr}

\medskip

K.~Jabbusch, \textsc{Department of Mathematics \& Statistics, University of Michigan--Dearborn, 4901 Evergreen Rd, Dearborn, Michigan, 48128, USA}\par\nopagebreak
\textit{E-mail address}: \texttt{jabbusch@umich.edu}

\medskip

S.~Makarova, \textsc{Mathematical Sciences Institute, The Australian National University, Hanna Neumann Building \#145, Science Road, Canberra ACT 2601, Australia}\par\nopagebreak
\textit{E-mail address}: \texttt{svetlana.makarova@anu.edu.au}

\medskip

E.~Mazzon, \textsc{Centre de Mathématiques Laurent Schwartz (CMLS), CNRS, École polytechnique, Institut Polytechnique de Paris, Palaiseau, France}\par\nopagebreak
\textit{E-mail address}: \texttt{e.mazzon15@alumni.imperial.ac.uk}

\medskip

N.~Viswanathan, \textsc{School of Physical and Chemical Sciences, Queen Mary University of London, United Kingdom,  E1 4NS}\par\nopagebreak
\textit{E-mail address}: \texttt{n.viswanathan@qmul.ac.uk}
}}

\title{Birational \(\mathbb P^2\)-bundle structures \\
and toric 2-Fano manifolds}

\author{Carolina Araujo,
Roya Beheshti,
Ana-Maria Castravet,
Kelly Jabbusch,\\
Svetlana Makarova,
Enrica Mazzon,
Nivedita Viswanathan}

\date{}

\begin{document}

\maketitle

\begin{abstract} 
    In this paper, we develop a general birational program toward the classification of toric 2-Fano manifolds through the investigation of the minimal projective bundle dimension \(m(X) \in \{1,\dots,\dim(X)\}\) introduced in our previous work. 
    This invariant captures the minimal degree of a dominating family of rational curves on \(X\) and admits a natural combinatorial interpretation in terms of centered primitive collections.
    We establish a general obstruction to being 2-Fano in terms of the existence of projective bundle structures and prove a birational reduction theorem for toric Fano manifolds with \(m(X)=2\). This allows us to show that the only toric 2-Fano manifold \(X\) with \(m(X) = 2\) is \(X\cong \mathbb{P}^2\). 
    We also investigate how this strategy might extend to \(m\ge3\).
\end{abstract}

\setcounter{tocdepth}{1}
\tableofcontents

\newpage

% ---
\section{Introduction}
% ---

A \emph{Fano manifold} is a smooth projective variety with ample anticanonical divisor. 
Fano manifolds exhibit remarkable geometric and arithmetic properties. 
In particular, as first observed by Mori in his seminal paper \cite{mori79}, every Fano manifold is covered by rational curves.
Moreover, the classical theorem of Tsen guarantees that every one-dimensional family of Fano hypersurfaces admits a section. 
These results were later generalized.
Fano manifolds are known to be \emph{rationally connected} \cite{KMM92}, i.e., any two points can be connected by a rational curve, while any one-dimensional family of rationally connected manifolds admits a section \cite{GHS}.

In recent years, there has been growing interest in formulating higher-order analogues of the Fano condition. 
In this context, de Jong and Starr introduced in \cite{dJS06} the notion of a \emph{$2$-Fano manifold}: a Fano manifold $X$ whose second Chern character satisfies 
\[
\mathrm{ch}_2(X) \cdot S > 0 \quad \text{for every surface } S \subset X.
\]
Examples of $2$-Fano manifolds include projective spaces and smooth hypersurfaces of degree $d$ in $\PP^n$ with $d^2<n$.
It is expected that $2$-Fano manifolds should enjoy strengthened versions of the geometric and arithmetic properties of classical Fano manifolds. 
For instance, under suitable hypotheses, $2$-Fano manifolds are covered by rational surfaces (\cite{dJS07hF_and_rat_surfaces}, \cite{AraujoCastravet2012}), and it is expected that every two-dimensional family of $2$-Fano manifolds admits a rational section, up to Brauer obstruction.

The $2$-Fano condition is extremely restrictive, and one of our long-term goals is the complete classification of $2$-Fano manifolds. 
Several special classes have already been classified: $2$-Fano varieties of high index in \cite{AraujoCastravet2013}, homogeneous $2$-Fano varieties in \cite{team2022}, and horospherical $2$-Fano varieties in \cite{AraujoCastravet2024}.

The present article concerns the problem of classifying toric $2$-Fano manifolds.
After extensive work on this problem (\cite{Nobili2011}, \cite{Nobili2012}, \cite{Sato2012}, \cite{Sato2016}, \cite{SatoSuyama2020}, \cite{SanoSatoSuyama2020},  \cite{Shrieve2020}, \cite{team2023}), 
 projective spaces remain the only known examples of toric $2$-Fano manifolds, leading to the following conjecture.

\begin{conjecture}[{\cite[Conjecture 4.3]{SanoSatoSuyama2020}}] \label{conjecture}
The only toric 2-Fano manifolds are projective spaces.
\end{conjecture}

Toric geometry provides an effective bridge between algebraic geometry and combinatorics: there is a one-to-one correspondence between $n$-dimensional toric varieties and lattice fans in $\mathbb{R}^n$, and many geometric properties of a toric variety can be translated into combinatorial statements about the associated fan.
A new approach to the classification of toric $2$-Fano manifolds was introduced in \cite{team2023}, and the goal of this paper is to establish the conjecture for the first case beyond those treated in \cite{team2023}.
This case exhibits substantially new birational phenomena and requires new techniques.
As proposed in \cite{AraujoCastravet2012}, the idea is to investigate $2$-Fano manifolds through their \emph{minimal families of rational curves}. 
By \cite{CFH}, minimal families of rational curves on an $n$-dimensional proper toric manifold $X$ correspond to \emph{primitive relations} of the form
\begin{equation}\label{CSR}
    x_0 + \dots + x_m = 0,
\end{equation}
satisfied by some of the generating vectors $x_i \in \mathbb{R}^n$ of the fan of $X$ (see Section~\ref{section:background} for details).  
A primitive relation of the form \eqref{CSR} corresponds to a $\mathbb{P}^m$-bundle structure on an open subset of $X$.
In \cite{team2023}, we introduced the invariant
\[
m(X) 
= 
\min\Big\{ m \in \mathbb{Z}_{>0} \ \Big| \ \text{there exists a primitive relation as in \eqref{CSR}} \Big\}
\ \in\ \{1,\dots,n\},
\]
called the \emph{minimal $\PP$-dimension} (or \emph{minimal projective bundle dimension}) of $X$, and proved that a toric Fano manifold $X$ with $m(X)=1$ is not $2$-Fano. 
We also classified all $n$-dimensional toric Fano manifolds $X$ with $m(X)\ge n-2$, and showed that $\mathbb{P}^n$ is the only $2$-Fano manifold among them.
In this paper, we develop a general birational program toward \cref{conjecture} and prove the following:

\begin{theorem}\label{main_thm}
Let $X$ be a toric Fano manifold with $\dim(X)>2$ and $m(X)=2$. 
Then $X$ is not $2$-Fano.    
\end{theorem}

\cref{main_thm} substantially narrows the range in which a counterexample to \cref{conjecture} could occur. Indeed, combining this theorem with prior work (\cite{Cas06}, \cite{Sato2012},  \cite{AN2016}, \cite{Sato2016}, \cite{SatoSuyama2020}, \cite{SanoSatoSuyama2020}, \cite{team2023}), and noting that products of Fano manifolds are not 2-Fano, we get the following classification of 2-Fano toric manifolds:
\begin{corollary}
    If \(X\) is a toric 2-Fano manifold of dimension \(n\), then either \(X \cong \PP^n\), or, if it exists, \(X\) satisfies the following conditions on its dimension, minimal \(\PP\)-dimension and Picard number:
    \[
        n \geq 9, \quad
        3 \leq m(X) \leq n-3, \quad
        4 \leq \rho_X < 2n - \frac{1}{30} \left( \sqrt{60n+1249} - 37 \right).
    \]
\end{corollary}

The proof of \cref{main_thm} builds on the strategy proposed in \cite{team2023}.
Let us briefly recall it here and refer to Section~\ref{section:strategy} for more details.
As noted above, if $X$ is a toric Fano manifold with $m(X)=m$, then $X$
contains a dense open subset $U\subset X$ admitting a $\PP^m$-bundle structure $\pi \colon U\to W$. 
If $\dim(X)>m$ and $X\setminus U$ has codimension at least $2$, then $W$ contains a complete rational curve (see \cref{lemma:existence of curve}), and one concludes that $X$ is not $2$-Fano by the following general result: 

\begin{theorem}[See \cref{thm_existence_of_surface}]\label{thm_general_obstruction}
Let $X$ be a complex manifold admitting a $\PP^m$-bundle structure $\pi \colon U\to W$ ($m\geq 1$) on a dense open subset $U\subset X$, and assume that $W$ contains a complete rational curve $C\subset W$. Then $X$ contains a proper surface $S$ such that
\[ \operatorname{ch}_2(X)\cdot S\le0. \]
\end{theorem}

On the other hand, if $X\setminus U$ has codimension $1$, then the idea is to proceed in three steps:
\begin{itemize}
    \item[] {\bf Step 1.} Construct a birational map $f \colon X \dashrightarrow Y$ 
    transforming $X$ into a proper toric variety $Y$ admitting a $\PP^m$-bundle structure on an open subset $V\subset Y$,
    such that $Y\setminus V$ has codimension at least $2$ in $Y$. 

    \item[] {\bf Step 2.} Construct a surface $S \subset V$ such that $\ch_2(Y)\cdot S \leq 0$.

    \item[] {\bf Step 3.} Compare the intersection numbers $\ch_2(Y)\cdot S$ and $\ch_2(X)\cdot \tilde S$, where $\tilde S \subset X$ is the proper transform of $S$, to show that $\ch_2(X)\cdot \tilde S \leq 0$.  
    For this step, one needs an explicit description of the map $f \colon X \dashrightarrow Y$ constructed in Step~1 and a suitable choice of surface $S$ in Step~2.
\end{itemize}
In \cite{team2023} this strategy was carried out in the simplest case $m=1$.
The case $m=2$ requires a significantly more refined argument: whereas for $m=1$,
Step~1 can be achieved with at most two smooth blowdowns, remaining within the projective category, the case $m=2$
requires flips in addition to blowdowns and may lead to a non-projective toric manifold $Y$ (see \cref{remark:nonprojective_may_occur}).
After establishing the birational strategy in the case $m=2$, which is the first setting where flips are unavoidable, we investigate how the strategy might extend to \(m\ge3\) in \cref{section: complications for m=3}.

The paper is organized as follows.  
In Section~\ref{section:background}, we review the necessary background on toric varieties, focusing on primitive relations, and discuss the geometric significance of the invariant  $m(X)$. 
In Section~\ref{section:strategy}, we describe the strategy outlined above in more detail and prove a few technical lemmas.
In Section~\ref{sec:construction maps}, we accomplish Step~1 in the case $m(X)=2$ by proving a birational reduction theorem that may be useful beyond the 2-Fano problem (\cref{theorem: construct blowdowns and flips}).
Namely, we construct a sequence of blowdowns and flips
\[
X \to X' \dashrightarrow Y \, ,
\]
where $Y$ is a proper toric manifold admitting a $\PP^2$-bundle structure on the complement of a closed subset of codimension at least $2$. 
In Section~\ref{section:step2}, we prove \cref{thm_existence_of_surface}, showing that Step~2 can be achieved for any value of $m(X)$.
Finally, in \cref{section: Chern character computations}, we compute how Chern characters change under blowdowns and flips, and show that $\ch_2(X)\cdot \tilde S \leq \ch_2(Y)\cdot S$, where $\tilde S \subset X$ is the proper transform of $S$.
In \cref{section: complications for m=3}, we indicate which parts of the proof of Step 1 extend to the case \(m(X) \geq 3\), highlight where complications arise, and outline possible approaches to address them.

\paragraph{Notation and conventions.}
We work over \(\CC\).
A \emph{toric manifold} means a smooth toric variety over $\CC$.  We denote by $\PP(E)$ the projective bundle \(\text{Proj}(\text{Sym}(E))\).
The Graded Ring Database \cite{GradedRingDatabase} collects the combinatorial data of toric Fano manifolds of dimension up to 6.
This data is also accessible via NormalToricVarieties package \cite{NormalToricVarietiesPackage} in Macaulay2 \cite{M2}.
We refer to specific examples using the way one can call them in Macaulay2, namely \textit{(dimension, Macaulay2 ID)}.

\paragraph{Acknowledgments.}
A great part of this work was developed during visits to Simons Laufer Mathematical Sciences Institute (formerly Mathematical Sciences Research Institute) as part of the ``Summer Research in Mathematics'' program in July 2023 (supported by the National Science Foundation under Grant No. DMS-1928930 and the National Security Agency under Grant No. H98230-23-1-0004) and to the Institut Henri Poincar\'e in March 2024.
We thank SLMath and the IHP for the financial support and the great working conditions provided to us during our visit, and we thank SLMath for the post-program support provided for further meetings.  We are very grateful to Cinzia Casagrande and Benjamin Nill for many informative discussions about  toric varieties.

Carolina Araujo was supported by grants from CNPq, Faperj and CAPES/COFECUB. Roya Beheshti was partially supported by NSF grant DMS-2101935 and Simons Foundation
grant 00007742. Ana-Maria Castravet was supported by ANR grant FanoHK and the Institut Universitaire de France. 
Nivedita Viswanathan was partially supported by the Engineering and Physical Sciences Research Council (EPSRC) Grant EP/V048619/1 ``K\"ahler-Einstein metrics on Fano Manifolds'' and EP/V056689/1 ``The Calabi Problem for smooth Fano threefolds''.

\paragraph{AI statement.}
AI was used to proofread and polish the final draft
and to find an example for \cref{remark:nonprojective_may_occur} from the database provided by Will Reynolds (see the appendix of \cite{team2023}).
AI was not used to produce results nor \LaTeX{} code.

% ---
\section{Background on toric geometry}
\label{section:background}
% ---

In this section, we recall some important facts about toric geometry and fix the notation to be used throughout the paper.
We refer to \cite{Fulton} and \cite{CoxLittleSchenck2011} for the general theory of toric varieties. 

A toric variety is a normal $n$-dimensional complex variety $X$ that contains the torus $(\mathbb{C}^*)^n$ as a dense open subset, together with an action of $(\mathbb{C}^*)^n$ extending its natural action on itself. 
There is a one-to-one correspondence between $n$-dimensional toric varieties and lattice fans $\Sigma$ in $\mathbb{R}^n$.
Given a lattice fan $\Sigma$ in $\mathbb{R}^n$, we denote by $G(\Sigma)\subset N=\ZZ^n$ the set of generators of the $1$-dimensional cones in $\Sigma$.
Likewise, for each cone $\sigma\in \Sigma$, we denote by $G(\sigma)\subset G(\Sigma)$ the set of generators of the rays of $\sigma$.

\begin{notation}
    In order to simplify the notation, we write
    $\bar v$ to denote a tuple or set of vectors $\{ v_1, \dots, v_r \}$, and \(\sum \bar v\) to denote the sum $v_1 + \dots + v_r$.
    If $\bar\mu=( \mu_1 , \dots , \mu_s )$ is a tuple of $s$ integers
    and $\bar w$ is a tuple of $s$ vectors, we write $\bar \mu \cdot \bar w := \mu_1 w_1 + \dots + \mu_s w_s$.
    Finally, we will occasionally write \(\{v_1, \dots,\check{v_i}, \dots , v_r\}\) for \(\bar v \setminus \{v_i\}\).
\end{notation}

Given a toric variety $X$, we denote by $\Sigma_X$ the fan associated to $X$.
We recall that $X$ is \emph{smooth} if and only if the fan $\Sigma_X$ is \emph{unimodular}, i.e., for each cone $\sigma\in \Sigma_X$, the set of generators $G(\sigma)$ is part of a basis of $N$.
The variety $X$ is \emph{proper} if the fan $\Sigma_X$ is \emph{complete}, i.e., its support is the whole $\mathbb{R}^n$.
Finally, we recall that there is a one-to-one inclusion-reversing correspondence between cones in $\Sigma_X$ and orbit closures in $X$ for the torus action.
Given a cone $\sigma\in \Sigma_X$, we write $V(\sigma) \subset X$ for the corresponding orbit closure,
or $V(\bar v)$ when $G(\sigma) = \bar v$. One has $\dim(\sigma) = \codim_X V(\sigma)$.

\subsection{Primitive collections and relations}
\label{subsection_PC}

Let $X$ be an $n$-dimensional smooth and proper toric variety, and denote by $\Sigma=\Sigma_X$ the associated fan in $\mathbb{R}^n$.
We denote by $A_1(X)$ the group of algebraic $1$-cycles on $X$ modulo numerical equivalence.
Recall that there is an exact sequence (\cite[Proposition 6.4.1]{CoxLittleSchenck2011}):
\[
0 \ \longrightarrow \ A_1(X) \ \xrightarrow{\ \alpha \ } \ \ZZ^{G(\Sigma)} \ \xrightarrow{\ \beta \ } \ N \ \longrightarrow \ 0 \ ,
\]
where $\alpha\big([C]\big)= \big( C \cdot V(v)\big)_{v \in G(\Sigma)}$, and $\beta\big((\mu_v)_{v \in G(\Sigma)}\big)=\sum_{v \in G(\Sigma)} \mu_v v$. 
Therefore, the elements of $A_1(X)$ can be identified with integral relations among the vectors in $G(\Sigma)$,
and one can investigate and characterize geometric properties of curves in terms of combinatorial properties of these relations.

\begin{definition}[{\cite[Definition 2.6]{Bat91}}]\label{def:PC&PR}
A \emph{primitive collection} of $X$ is a nonempty set $P=\{v_1, \dots, v_r\} \subseteq G(\Sigma)$ such that 
\[
\langle v_1, \dots , v_r \rangle \notin \Sigma \quad \quad \text{and} \quad \quad \langle v_1, \dots,\check{v_i}, \dots , v_r \rangle \in \Sigma \quad \forall i=1,\dots, r.
\]
The \emph{focus} of the primitive collection $P$ is the minimal cone $\sigma(P)= \langle w_1, \dots , w_s \rangle \in \Sigma$ such that $v_1 + \cdots + v_r \in \sigma(P)$. 
There is therefore a relation
\[
r(P)\colon \,v_1 + \cdots + v_r = \mu_1 w_1 + \cdots + \mu_s w_s,
\]   
with $\mu_j \in \ZZ_{>0}$ for $j=1, \dots, s$. 

We call  $r(P)$ the \emph{primitive relation} associated to $P$. 
The \emph{order} of $P$ is $|P|=r$, and the \emph{degree} of $P$ is $\deg(P)=r- \sum_{j=1}^{s}\mu_j$.
We note that $\deg(P) = -K_X\cdot r(P)$, i.e., it coincides with the anticanonical degree of the curve associated to \(r(P)\), by \cite[Theorem 8.2.3 and Eq. (6.4.7)]{CoxLittleSchenck2011}.

The primitive collection $P$ is called \emph{centered} if $\sigma(P)=\{0\}$, i.e.,
\[
    r(P) \colon \, v_1 + \cdots + v_r=0.
\]

We denote by $\PC(\Sigma)$ or $\PC(X)$ the set of primitive collections of $X$, and by $\PR(\Sigma)$ or $\PR(X)$ the set of primitive relations of $X$.
\end{definition}

\begin{remark}
\label{remark: PC does not intersect focus}
By \cite[Proposition 3.1]{Bat91}, for any primitive collection $P$ on a proper toric manifold, we have $P \cap \sigma(P) = \varnothing$. In the notation above, $\{v_1, \dots, v_r\} \cap \{w_1, \dots ,w_s\} = \varnothing$.
\end{remark}

\begin{remark}
By \cite[Theorem 2.15]{Bat91}, when \(X\) is projective, the primitive relations generate the Mori cone of $X$, which is the cone $\NE(X)\subset A_1(X) \otimes_{\ZZ} \mathbb{R}$ generated by the classes of effective curves:
\[
\NE (X) = \sum_{P \in \PC(X)} \mathbb{R}_{\geq 0}\,r(P).
\]
\end{remark}

Next we discuss the notion of \emph{contractible} curve classes and their characterization in terms of primitive relations. 

\begin{definition}\label{def: contractible}
A primitive curve class $\gamma\in A_1(X)$ is said to be \emph{contractible} if there exists some irreducible curve having numerical class in $\mathbb{R}_{\geq 0}\gamma$, and there exists a proper toric variety $X_\gamma$ and an equivariant surjective morphism $\varphi_\gamma \colon X \to X_\gamma$ with connected fibers such that, for every irreducible curve $C \subset X$, 
\[\varphi_\gamma(C)= \{pt\} \Longleftrightarrow [C] \in \mathbb{R}_{\geq 0}\gamma.\]
We say that $\varphi_\gamma \colon X \to X_\gamma$ is the \emph{contraction} associated to the contractible curve class $\gamma$.
\end{definition}

\begin{proposition}[{\cite[Theorem 2.2]{Casagrande2003}}]
\label{prop: contractible}
Let $X$ be a proper toric manifold and let $\gamma \in A_1(X)$ be a primitive curve class.
Then  $\gamma$ is contractible if and only if $\gamma$ is a primitive relation of the form 
    \[ r(P) \colon\, v_1 + \dots + v_r = \mu_1 w_1 + \dots +\mu_s w_s,
    \]
    and the following condition holds. For every cone $\tau= \langle z_1, \dots, z_\ell \rangle \in \Sigma_X$ (possibly $\tau=\{0\}$) such that 
    \[\begin{cases}
        \{z_1,\ldots,z_\ell\}\cap \{v_1,\ldots, v_r,w_1, \ldots,w_s\} = \varnothing \text{ and} \\
        \langle w_1, \dots, w_s, z_1, \dots, z_\ell \rangle \in \Sigma_X,
    \end{cases}\] 
    and for each $i=1, \dots, r$, the following is also a cone of $\Sigma_X$:
    \[ 
        \langle v_1, \dots, \check{v_i}, \dots, v_r,  w_1, \dots, w_s, z_1, \dots, z_\ell  \rangle \, .
    \] 
\end{proposition}

\begin{remark}
A primitive integral class generating an extremal ray of $\NE(X)$ is called an \emph{extremal class} and is necessarily contractible. 
However, even when $X$ is a projective variety, there may exist contractible curve classes that are not extremal. In such cases, the targets of the associated contractions are not projective. 

When $X$ is a toric Fano manifold, any primitive relation of $X$ having degree $1$ is extremal (\cite[Corollary 4.4]{Casagrande2003}). 
\end{remark}

% -
\subsection{Primitive collections and relations under birational transformations}
\label{subsection: contractibility under birational transformations}
% -

In this subsection, we recall the behavior of primitive collections and contractible classes under smooth toric blowups and blowdowns, and describe the flips used later in the paper.
We start by recalling how the fan of a toric variety changes under a blowup.

\begin{proposition}[{\cite[Definition 3.3.17 and p.~133]{CoxLittleSchenck2011}}]
\label{proposition: fan of a blowup}
Let $Y$ be a toric manifold and let $\bar a$ be a subset of $G(\Sigma_Y)$ that spans a cone $\langle \bar a \rangle\in \Sigma_Y$.
Let $X$ be the toric manifold obtained from $Y$ by blowing up $V(\bar a)$.
Then $G(\Sigma_X) = G(\Sigma_Y) \cup \{b\}$, where $b = \sum \bar a$, and $\Sigma_X$ consists precisely of the following cones:
    \begin{itemize}
        \item $\sigma \in \Sigma_Y$ such that $\bar a \not\subseteq G(\sigma)$,
        \item $\langle \bar v , b \rangle$, for any $\sigma \in \Sigma_Y$ containing $\bar a$ and any $\bar v$ such that $\bar a \not\subseteq \bar v \subseteq G(\sigma)$.
    \end{itemize}
    In particular, it follows from this description that if $\langle \bar v, b \rangle$ is a cone in $\Sigma_X$, then $\langle \bar v, \bar a \rangle$ is a cone in $\Sigma_Y$.
\end{proposition}

The next two results of Sato describe the behavior of primitive collections under smooth toric blowdowns and blowups.

\begin{proposition}[{\cite[Corollary 4.9]{Sato2000}}]
\label{prop:PC_blowdown}
Let $X$ be a proper toric manifold, and let $f \colon X \rightarrow Y$ be the contraction associated to a contractible class in $\mathrm{NE}(X)$, corresponding to a primitive relation of the form 
\[r(Q)\colon t_1 + \dots+ t_s = z.\]
Then the fan $\Sigma_Y$ is obtained from $\Sigma_X$ by removing the ray generated by $z$, and
$X$ is the blowup of $Y$ along $V( t_1,\dots,t_s)$.
Furthermore, the primitive collections of $Y$ are precisely the following $P_Y\in \PC(Y)$:
\begin{itemize}
    \item $P_Y=P_X$ for some $P_X \in \PC(X)$ such that $z \notin P_X$ and $P_X \neq Q=\{t_1, \dots, t_s\}$;
    \item $P_Y=(P_X \setminus \{z\}) \cup \{t_1,\dots,t_s\}$ for some $P_X \in \PC(X)$ such that $z \in P_X$ and $(P_X \setminus \{z\}) \cup S \notin \PC(X)$ for any proper subset $S \subsetneq \{t_1,\dots,t_s\}$.
\end{itemize}
\end{proposition}

\begin{remark} \label{rem:class contracted curve}
    In the setup of \cref{prop:PC_blowdown}, for any $y\in V( t_1,\dots,t_s)$, $f^{-1}(y)\cong \PP^{s-1}$. The primitive relation $r(Q)$ is precisely the class in $A_1(X)$ of a line on such $\PP^{s-1}$.
\end{remark}

\begin{proposition}[{\cite[Theorem 4.3]{Sato2000}}]
\label{prop:PC_blowup}
Let $X$ be a toric manifold, and let $f \colon X \rightarrow Y$ be the contraction associated to a contractible class in $\mathrm{NE}(X)$, corresponding to a primitive relation of the form 
$r(Q)\colon \sum \bar t = z$.
Then the primitive collections of $X$ are precisely the following $P_X\in \PC(X)$:
\begin{itemize}
    \item $P_X = \bar t$;
    \item $P_X=P_Y$ for some $P_Y \in \PC(Y)$ such that $\bar t \not\subseteq P_Y$;
    \item minimal elements (with respect to inclusion) among sets of the form $P_X=(P_Y \setminus \bar t) \cup \{z\}$ for some $P_Y \in \PC(Y)$ such that $P_Y \cap \bar t \neq \varnothing$.
\end{itemize}
\end{proposition}

We recall the following results from \cite{Casagrande2003} on the behavior of contractible primitive relations under smooth toric blowups and blowdowns.

\begin{proposition}[{\cite[Lemma 5.4 and Lemma 5.5(i)]{Casagrande2003}}]
\label{proposition: conractibility on the blowup}
    Let $X$ be a proper toric manifold, $\langle\bar c\rangle\in\Sigma_X$ a cone, and
    $\tilde X \to X$ the blowup of $V_X(\bar c)$.
    Let $P \in \PC(X)$ be a primitive collection.
    \begin{enumerate}[label=(\roman*)]
        \item \label{item: contractibility type a}
        If $P \cap \bar c = \varnothing$, then $P$ is a primitive collection in $\tilde X$, and $r_{\tilde X} (P)$ is contractible in $\tilde X$ if and only if $r_X (P)$ is contractible in $X$.
        \item \label{item: contractibility type c}
        If $P \cap \bar c \neq \varnothing$ and $\bar c \not\subseteq P$, then $P$ is a primitive collection in $\tilde X$, and
        $r_{\tilde X} (P)$ is contractible in $\tilde X$ if and only if
        $r_X(P)$ is contractible on $X$ and $\langle \sigma_X(P), \bar c \rangle$ is not a cone of $X$. The latter condition means that $V_X(\sigma_X(P))$ does not intersect the center $V_X(\bar c)$ of the blowup.
    \end{enumerate}
\end{proposition}

\begin{corollary}
\label{lem:contractible after blowdown}
    Let $X$ be a proper toric manifold, $\langle\bar c\rangle\in\Sigma_X$ a cone, and
    $\tilde X \to X$ the blowup of $V_X(\bar c)$.
    Let \(P\subseteq G(\Sigma_X)\) be a primitive collection of both \(X\) and \(\tilde X\) such that the relation \(r_{\tilde X} (P)\) is contractible in \(\tilde X\).
    Then \(r_X(P)\) is contractible in \(X\).
\end{corollary}

\begin{proof}
    One can notice that our assumptions eliminate the case \(\bar c \subseteq P\), hence the claim is a special case of \cref{proposition: conractibility on the blowup}.
\end{proof}

We end this subsection by describing the flips that appear in our context.

\begin{setup}
\label{setup_flips}
Let $X$ be a proper toric manifold. Let $Q$ be a primitive collection with a contractible primitive relation of the form
    \[r_X(Q)\colon \, a_1+\dots+a_m=c_1+\dots+c_\ell, \]
with $m, \ell \geq 2$  and  all the $a_i$'s and $c_j$'s distinct.
Let $\tilde X\to X$ be the blowup of $X$ along $V(\bar c)$, with new vector $z=\sum\bar c = \sum\bar a \in G(\Sigma_{\tilde X})$.
By \cref{proposition: conractibility on the blowup}, $\sum\bar a = z$ is a contractible primitive relation on $\tilde X$. It gives rise to a smooth blowdown $\tilde X\to Y$.
The \emph{flip corresponding to $r_X(Q)$} is the rational map $X \dashrightarrow Y$ obtained by composing the blowup and blowdown associated to the primitive relations $\sum\bar c = z$ and $\sum\bar a = z$ on $\tilde X$, respectively:
    \[\begin{tikzcd}
    	& {\tilde X} \\
    	X && Y.
    	\arrow["{\sum\bar c = z}"', from=1-2, to=2-1]
    	\arrow["{\sum\bar a = z}", from=1-2, to=2-3]
    	\arrow[dashed, from=2-1, to=2-3]
    \end{tikzcd}\]

Let us describe the modification of the fans under the flip $X \dashrightarrow Y$.
Any cone $\sigma\in \Sigma_X$ not containing $\bar c$ remains a cone in $\Sigma_Y$. 
Now consider 
\[
\mathcal L \ := \ 
\big\{
\sigma\in\Sigma_X \ \big| \
G(\sigma)\cap(\bar a\cup\bar c)=\varnothing \, ,\quad 
\langle\bar c,\sigma\rangle\in\Sigma_X
\big\}.
\]
Since $r_X(Q)$ is contractible, by \cref{prop: contractible}, for every \(\sigma\in\mathcal L\) and every \(u\in \bar a\), we have 
\[
\langle \bar a\setminus\{u\},\bar c,\sigma\rangle \in \Sigma_{X}.
\]
For each \(\sigma\in\mathcal L\), the flip replaces the subdivision of the cone $\langle \bar a,\bar c,\sigma\rangle$ given by 
\[
\bigcup_{u\in \bar a}
\langle \bar a\setminus\{u\},\bar c,\sigma\rangle
\]
with the subdivision given by 
\[
\bigcup_{v\in \bar c}
\langle \bar a,\bar c\setminus\{v\},\sigma\rangle.
\]

\end{setup}

% -
\subsection{The minimal \texorpdfstring{$\PP$}{P}-dimension}
\label{subsection: minimal P-dimension}
% -

By \cite[Proposition 3.2]{Bat91}, a projective toric manifold always admits a centered primitive collection
(see also \cite[Corollary 3.3]{CFH} for a geometric proof).
This is not true in general for a proper toric manifold \cite[Example 3.4]{CFH}.
For a projective toric manifold $X$, we defined in \cite[Definition 2.10]{team2023} the \emph{minimal $\PP$-dimension} $\minRC(X)$ of $X$ as the minimal integer $m$ such that $X$ has a centered primitive relation of order $m+1$:
\[
    x_0 + x_1 + \dots + x_m = 0.
\]
One has $m(X)\in\{1,\dots,\dim(X)\}$, and $m(X)=\dim(X)$ if and only if $X$ is the projective space.

By \cite[Corollary 2.6.]{CFH}, a proper toric manifold $X$ has a centered primitive relation of order $m+1$ if and only if it admits a $\PP^m$-bundle structure on a dense open subset. We recall the explicit description of the biggest such open subset from \cite{team2023}.

\begin{notation} \label{not:no proj bundle}
Let $X$ be a proper toric manifold, and $P=\{x_0,\dots, x_k\}\in \PC(X)$ a centered primitive collection. Denote by $\mathcal{E}_P$ the set of cones $\sigma = \langle v_1, \dots, v_r \rangle \in \Sigma_X$ such that $P \cap G(\sigma) = \varnothing$,
and $\{v_1, \dots, v_r, x_{j_1}, \dots, x_{j_s} \} \in \PC(X)$ for some $s$, \(1\leq s\leq k\), i.e.,
\[ \cE_P \coloneqq \left\{ \sigma \in \Sigma_X
\mid P \cap G(\sigma) = \varnothing \text{ and }
\exists P' \subsetneq P \text{ such that } P' \cup G(\sigma) \in \PC(X)
\right\}.
\]
Set
\[  V(\mathcal{E}_P) \coloneqq \bigcup_{\sigma \in \mathcal{E}_P} V(\sigma) \subset X. \]
\end{notation}

\begin{proposition}[{\cite[Proposition 2.14]{team2023}}]
\label{proposition: Pk bundle}
Let $X$ be a proper toric manifold admitting a centered primitive collection $P=\{x_0,\dots, x_m\}\in \PC(X)$, and let 
$V(\mathcal{E}_P)$ be as in Notation~\ref{not:no proj bundle}.
Then the open subset $U= X \setminus V(\mathcal{E}_P)$ admits a $\PP^m$-bundle structure over a toric manifold.
\end{proposition}

The following table was obtained in \cite{team2023} and lists the number of toric Fano manifolds of dimension $4$, $5$ and $6$ according to their minimal $\mathbb{P}$-dimension.

\begin{table}[h]
\centering
\begin{tabular}{|l||l|l|l|l|l|l|l|}
\hline
\small{$\dim(X)$} & \small{\# Fano's}  & \small{\#($m$=1)} & \small{\#($m$=2)} & \small{\#($m$=3)} & \small{\#($m$=4)} & \small{\#($m$=5)} & \small{\#($m$=6)} \\ \hline \hline
4   & 124  & 107    & 15     & 1      & 1      &       &        \\ \hline
5   & 866  & 744    & 112    & 8      & 1      & 1     &        \\ \hline
6   & 7622 & 6333   & 1174   & 105    & 8      & 1     & 1      \\ \hline
\end{tabular}
\end{table}

In \cite{team2023}, we classified $n$-dimensional toric Fano manifolds $X$ with $m(X)\ge n-2$. 
As suggested by the table above, for $n\geq 5$, there is a unique $n$-dimensional toric Fano manifold with $m(X)= n-1$, namely, the blowup of $\PP^n$ along a linear $\PP^{n-2}$, and there are exactly $8$ isomorphism classes of $n$-dimensional toric Fano manifolds with $m(X)= n-2$, explicitly described in \cite[Theorem 1.4]{team2023}.

% ---
\section{The general strategy}
\label{section:strategy}
% ---

In this section, we describe the general strategy proposed in \cite{team2023} to approach \cref{conjecture}. 
Let $X$ be a toric Fano manifold with minimal $\PP$-dimension $m\geq 1$.
Recall that $m=\dim(X)$ if and only if $X$ is the projective space, and in this case $X$ is $2$-Fano if $\dim(X)\geq 2$.
So from now on we assume that $m<\dim(X)$ and 
let $P = \{ x_0 , x_1, \dots , x_m \}$ be a centered primitive collection of $X$. 
We denote by $U=X\setminus V(\mathcal{E}_P)$ the open subset from \cref{proposition: Pk bundle}, 
which admits a $\PP^m$-bundle structure associated to $P$. 
If $V(\mathcal{E}_P)$ happens to be of codimension at least $2$, then 
it is possible to find a proper surface $S\subset U$ such that $\ch_2(X)\cdot S\leq 0$ (see \cref{corollary: existence of surface with nonpositive intersection}), showing that $X$ is not $2$-Fano. 
When $\codim V(\mathcal{E}_P) = 1$, our aim is to perform a sequence of birational transformations, each removing a divisorial component from $V(\mathcal{E}_P)$, which results in increasing the codimension of $V(\mathcal{E}_P)$.
To make this precise, we introduce the following definition.
\smallskip

\begin{definition}
    Let $X$ be a proper toric manifold with a centered primitive collection $P = \{ x_0 , x_1, \dots , x_m \}$.
    We say that a primitive collection $Q \in \PC(X)$ is \emph{relevant} if there exist $P' \subsetneq P$ and $a\in G(\Sigma_X) \setminus P$ such that $Q = P' \cup \{a\}$.
    The corresponding primitive relation $r(Q)$ will be called a \emph{relevant primitive relation}.
    We will denote by $\RPC(X)$ the set of relevant primitive collections of $X$, and by $\RPR(X)$ the set of relevant primitive relations.
\end{definition}

\begin{remark}
    The notion of being relevant depends on the chosen centered primitive collection.
    However, we suppress this choice from notation because we always work with one fixed centered primitive relation.
\end{remark}

The following statement follows immediately from the definitions.

\begin{lemma}
    Let $X$ be a proper toric manifold with a centered primitive collection $P = \{ x_0 , x_1, \dots , x_m \}$.
    Then there is a surjection
    \[
\begin{array}{ccc}
   \RPC(X) & \to & \big\{ \text{divisorial components of } V(\cE_P) \big\} \\
 P' \cup \{a\} & \mapsto & V(a). \hfill
\end{array}
    \]
    In particular, $\codim V(\cE_P) \geq 2$ if and only if $\RPC(X) = \varnothing$. \hfill \qed
\end{lemma}

We now sketch our general strategy to prove that $X$ is not $2$-Fano:
\begin{itemize}
    \item[] {\bf Step 1.} We perform blowdowns and flips $f \colon X \dashrightarrow Y$ to arrive at a proper toric manifold $Y$ such that $P$ is still a primitive collection of $Y$ and $\RPC(Y) = \varnothing$.

    \item[] {\bf Step 2.} We construct a surface $S \subset Y \setminus V_Y(\cE_P)$ such that
    $\ch_2(Y)\cdot S \leq 0$.

    \item[] {\bf Step 3.} We show that for a suitably chosen surface \(S\) in Step 2, its proper transform $\tilde S \subset X$ witnesses that $X$ is not $2$-Fano.
    More precisely: $\ch_2(X)\cdot \tilde S \leq \ch_2(Y)\cdot S \leq 0$.  
\end{itemize}

In \cite{team2023}, we established that this strategy works for $m=1$.
In the present paper, we show that Step 2 holds in general (see \cref{corollary: existence of surface with nonpositive intersection}), and prove that Steps 1 and 3 can be carried out when $m=2$ (see \cref{theorem: construct blowdowns and flips} for Step 1 and \cref{section: Chern character computations} for Step 3).

% -
\subsection{Detailed strategy in the case \texorpdfstring{$m(X)=2$}{}}
\label{subsection: general strategy}
% -

Let $X$ be a toric Fano manifold.
The most technically demanding part of the above outline is to get to the manifold $Y$ which is birational to $X$ and admits a $\PP^m$-bundle structure on a big open subset. We recall that in the case $m(X)=1$, $Y$ is obtained by performing at most two smooth blowdowns.
We now describe how to proceed when $m(X)=2$.
In this case, in addition to (at most three) smooth blowdowns, it may also be necessary to perform flips. 
Moreover, unlike the case $m(X)=1$, one may have to allow non-projective toric manifolds (see \cref{remark:nonprojective_may_occur}). 

Our first task is to understand the relevant primitive collections and relations that can occur in $X$. 
Fix a centered primitive relation 
$$x_0+x_1+x_2=0.$$
Since $X$ is Fano, the degree 
of any primitive relation is positive, and so,
a priori, we have the following four possibilities for relevant primitive collections and relations.
\begin{enumerate}
\item
\label{item: RPC x+a=b}
The primitive collection $\{x_i, a\}$, with primitive relation $x_i+a = b$;
\item 
\label{item: RPC x+x+a=b+c}
The primitive collection $\{x_i,x_j, a\}$, with primitive relation $x_i+x_j+a = b+c$, where $b$ and $c$ are distinct;
\item 
\label{item: RPC x+x+a=b}
The primitive collection $\{x_i,x_j, a\}$, with primitive relation $x_i+x_j+a = b$;
\item 
\label{item: RPC x+x+a=2b}
The primitive collection $\{x_i,x_j, a\}$, with primitive relation $x_i+x_j+a = 2b$.
\end{enumerate}

\begin{remark}
    We will show in \cref{lemma:fewer_possibilities} below that it is enough to consider cases \ref{item: RPC x+a=b} and \ref{item: RPC x+x+a=b+c} only.
    Furthermore, the vectors \(a,b,c\) appearing in the relevant primitive relations of types \ref{item: RPC x+a=b} and \ref{item: RPC x+x+a=b+c} above are distinct from \(x_0,x_1,x_2\) by \cref{lemma: RHS of an RPR is not contained in the centered PC}.
\end{remark}

\begin{remark}
\label{remark: when a may repeat}
    We note that we cannot have a type \ref{item: RPC x+a=b} relation \(x_i + a = b\) together with a type \ref{item: RPC x+x+a=b+c} relation \(x_i+x_j+a=c+d\), with repeating \(x_i\) and \(a\), because the former implies that \(\{x_i,a\}\) does not span a cone in \(\Sigma_X\), while the second implies that it does.
    However, we may have two relations of type \ref{item: RPC x+x+a=b+c} of the forms \(x_i+x_j+a=b+c\) and \(x_i+x_k+a=d+e\), with repeating \(x_i\) and \(a\). This appears, for instance, in Macaulay2 ID (6, 627).
\end{remark}

In what follows, we will need the following notion and result from \cite{team2023}.

\begin{definition}
Let $X$ be a proper toric manifold. 
A vector $v \in G(\Sigma_X)$ is an \emph{opponent} of $w \in G(\Sigma_X)$ if $\{v,w\}$ does not span a cone in \(\Sigma_X\).
\end{definition}

\begin{proposition}[{\cite[Corollary 2.23]{team2023}}]\label{prop:one_oponent}
    If $X$ is a toric Fano manifold with $m(X) \geq 2$, then each vector $v\in G(\Sigma_X)$ admits at most one opponent.
\end{proposition}

We make the following easy observation, which will be used several times in the arguments that follow.

\begin{remark}
\label{remark: cannot have relation a+b=2c+d with a b forming a cone}
    Assume that \(X\) is a proper toric manifold such that each vector $v\in G(\Sigma_X)$ admits at most one opponent. Assume that there exist distinct generators \(a,b,c,d \in G(\Sigma_X)\) such that \(\langle a,b\rangle\) is a cone in \(\Sigma_X\) and \(a+b=2c+d\).  
    Since two distinct cones of the same dimension must intersect along a proper face, we get that \(\{c,d\}\) is a primitive collection in \(X\).
    We claim that the degree of the corresponding primitive relation is at most \(0\).
    In particular, this situation cannot occur if \(X\) is Fano.
    To prove the claim, first notice that we cannot have \(c+d=0\), as otherwise we get \(a+b=c\) with vectors on both sides spanning cones.
    Moreover, if the corresponding primitive relation is \(c+d=e\), substituting it into the equality \(a+b=2c+d\) yields \(a+b=c+e\). As before, \(\{c,e\}\) must be a primitive collection, and both \(d\) and \(e\) are opponents of \(c\), which is a contradiction.
\end{remark}

\begin{lemma}\label{lemma:fewer_possibilities}
    Let $X$ be a toric manifold admitting a centered primitive relation of order $3$: $x_0 + x_1 + x_2 = 0$.
    \begin{enumerate}
        \item \label{item: ignore RPR x+x+a=b}
        If there is a relevant primitive relation of type \ref{item: RPC x+x+a=b}, say $x_i + x_j + a = b$, then $x_k + b = a$ is another relevant primitive relation on $X$, which is of type \ref{item: RPC x+a=b}.
        \item \label{item: no RPR x+x+a=2b}
        If $X$ is Fano, there are no relevant primitive relations of type \ref{item: RPC x+x+a=2b} on $X$.
    \end{enumerate}
\end{lemma}

\begin{proof}
    To prove part \ref{item: ignore RPR x+x+a=b}, we add the third vector $x_k$ to both sides of $x_i + x_j + a = b$ and obtain the equality $x_k + b = a$.
    Since two cones are only allowed to intersect along faces, $x_k$ and $b$ cannot form a cone, and we conclude that $x_k + b = a$ is a primitive relation.

    For part \ref{item: no RPR x+x+a=2b}, assume that there is a primitive relation of the form  $x_i + x_j + a = 2b$.
    We add $x_k$ to both sides again and obtain $x_k+2b = a$.
    Again, $x_k$ and $b$ cannot form a cone, so we get a primitive relation $x_k + b = c$.
    It is of degree $1$, hence extremal, so $\langle c, b \rangle$ is a cone of $X$.
    On the other hand, substitution yields the relation $c+b=a$, and so $\langle c, b \rangle$ is not a cone of $X$, a contradiction.
\end{proof}

\begin{remark}
\label{remark: at most 3 of type x+a=b}
Let $X$ be a toric Fano manifold with $m(X)=2$ and  a centered primitive relation $x_0 + x_1 + x_2 = 0$.
By \cref{prop:one_oponent}, each $x_i$ has at most one opponent, so we can have at most three relevant primitive relations of type \ref{item: RPC x+a=b}, which are then of the form
\[x_0+a_0 = b_0 \hspace{.5in} x_1+a_1 = b_1 \hspace{.5in} x_2+a_2 = b_2. \]
In particular, the $a_i$'s are pairwise distinct, and we will explain that the $b_i$'s are also pairwise distinct.
Assume that \(x_0 + a_0 = b\) and \(x_1 + a_1 = b\), then we have the (non-primitive) relation \(x_0 + x_1 + a_0 + a_1 = 2b\).
By adding \(x_2\) to both sides and using \(x_0+x_1+x_2=0\), we get \(x_2 + 2b = a_0 + a_1\), and \(\langle a_0,a_1 \rangle\) is a cone by uniqueness of opponents.
But this cannot happen by \cref{remark: cannot have relation a+b=2c+d with a b forming a cone}.

It may happen that $a_i = b_j$ for some $i\neq j$. 
If this is the case, say if we have two primitive relations
\[
    x_1 + a = b, \qquad
    x_2 + b = c,
\]
then adding these two together and adding \(x_0\) to both sides yields another relevant primitive relation of type (\ref{item: RPC x+a=b}): \(x_0+c=a\).
We need to treat this case separately, and therefore we give the following definition.
\end{remark}

\begin{definition}\label{def:exceptional_case}
    Let $X$ be a toric Fano manifold with $m(X)=2$ and a centered primitive relation $r(P) \colon x_0 + x_1 + x_2 = 0$.
    Assume that $X$ possesses the following primitive relations:
    \begin{equation}
    \label{eq: exceptional decomposition}
        r(P_0) \colon x_0 + c = a, \qquad
        r(P_1) \colon x_1 + a = b, \qquad
        r(P_2) \colon x_2 + b = c.
    \end{equation}
    We then say that the centered primitive relation
    $r(P)$ admits an \emph{exceptional decomposition}
    \[
        r(P) = r(P_0) + r(P_1) + r(P_2),
    \]
    and call $X$ with such a choice of centered primitive collection $P$ an \emph{exceptional case}.
\end{definition}

We are now ready to explain how to construct the sequence of birational maps in Step 1.

\begin{construction}\label{construction}
    We start with a toric Fano manifold $X$ with $\dim(X)>2$, $m(X) = 2$ and a fixed centered primitive relation $r(P) \colon x_0 + x_1 + x_2 = 0$.
    \begin{enumerate}
        \item If there are primitive relations of type \ref{item: RPC x+a=b}, we perform the blowdown or a composition of blowdowns associated to them. Denote the result by $X \to X'$.
        By \cref{remark: at most 3 of type x+a=b}, $X \to X'$ is a composition of at most $3$ blowdowns.

        \item
        In \cref{subsection: blow-downs of type 1}, we show that $P \in \PC(X')$, that no new relevant primitive collections appear on $X'$, and those that remain do not change their primitive relations.
        In particular, $X'$ does not have relevant primitive relations of type \ref{item: RPC x+a=b}.
        
        \item
        If $\RPR(X')= \varnothing$, then set $Y=X'$.
        If $\RPR(X')\neq \varnothing$, then, by \cref{lemma:fewer_possibilities}, $\RPR(X')$ only contains relevant primitive relations of type \ref{item: RPC x+x+a=b+c}.
        We show in \cref{lemma: flipping loci do not intersect} that, for every pair of distinct $r(Q), r(Q') \in \RPR(X')$,
        we have $V(\sigma(Q)) \cap V(\sigma(Q')) = \varnothing$, i.e., the flipping loci do not intersect.
        Thus we can flip all of them at once, and denote the result by $X' \dashrightarrow Y$.
        \item
        We then show in \cref{lemma: no new RPR after all flips} that {$P \in \PC(Y)$ and that} 
        $\RPC(Y) = \varnothing$, and so $P$ induces a $\PP^2$-bundle on an open subset of $Y$ whose complement $V_Y(\cE_P)$ has codimension $\geq 2$.
    \end{enumerate}
\end{construction}

\begin{remark}
\label{remark:nonprojective_may_occur}
The toric manifold $Y$ obtained from this construction need not be projective. This occurs, for instance, for the variety with Macaulay2 ID (6,798). This toric Fano 6-fold has a unique centered primitive collection of order 3 and exactly two relevant primitive relations, both of type 2. Flipping these relations produces a  non-projective toric manifold. This can be verified by observing that, after the flips, four nonzero effective curve classes add up to zero. 
\end{remark}

% -
\subsection{Relevant primitive relations after birational transformations}
% -

We end this section by describing what happens with a relevant primitive relation when we perform a blowdown as 
described in \cref{construction}(1).
The next lemmas will be used to verify the assertions made in  \cref{construction}(2).

\begin{lemma}
\label{lemma: RHS of an RPR is not contained in the centered PC}
    Let \(X\) be a proper toric manifold with a centered primitive relation \(x_0+\dots+x_m = 0\).
    Let \(\bar x = \overline{x'} \sqcup \overline{x''}\) be a decomposition of \(\bar x\) into two non-empty subsets.
    If \(\sum\overline{x'}+\sum \bar a= \sum \bar b\) is a contractible primitive relation for some \(\bar a, \bar b \subset G(\Sigma_X)\)
    such that \(\bar a \cap \bar x = \varnothing\),
    then \(\bar b\not\subset \bar x\)
    and \(\overline{x''} \not\subseteq \bar b\).
    In particular, if a primitive relation \(x_1+a=b\) is contractible, then \(b\notin \bar x\), and if \(x_1+\cdots+x_m+a= \sum \bar b\) is contractible, then \(x_0 \notin \bar b\).
\end{lemma}

\begin{proof}
    By \cref{remark: PC does not intersect focus}, \(\bar b \cap \overline{x'} = \varnothing \).
    Note that \(\langle \overline{x''}\rangle\) is a cone in \(\Sigma_X\), since by assumption, \(\overline{x''} \subsetneq \bar x\).
    By \cref{prop: contractible}, if \(\bar b \subseteq \overline{x''}\) or \(\overline{x''} \subset \bar b\), then \(\langle \overline {x'}, \overline{x''} \rangle = \langle \bar x \rangle\) is a cone in \(\Sigma_X\), which contradicts the assumption that \(\bar x\) is a primitive collection.
\end{proof}

\begin{lemma}
\label{lem:P and opponents after blowdown}
    Let $X$ be a proper toric manifold with a centered primitive relation $x_0+\dots+x_m=0$.
    Let $X \to X_\gamma$ be the morphism corresponding to a contractible primitive relation of the form
    $$\gamma = r_X(Q) \colon \, x_1+a=b.$$ 
    Then the following statements hold.
    \begin{enumerate}
        \item $\bar x \in \PC(X_\gamma)$.
        \item No new opponents appear in $X_\gamma$, i.e., if $P=\{u,v\} \in \PC(X_\gamma)$ then $P \in \PC(X)$.
        \item
        \label{item: new RPC on a contraction}
        If \(m=2\) and there is a relevant primitive collection $P$ in $X_\gamma$ that is not a primitive collection in $X$, then, for some $i\in\{0,2\}$, we have $\{x_i,b\} \in \RPC(X)$ and $P=\{x_1,x_i,a\}$.
    \end{enumerate}
\end{lemma}

\begin{proof}
    By \cref{prop:PC_blowdown}, the primitive collections of $X_\gamma$ are:
    \begin{align*}
     \PC(X_\gamma) & \ = \  \{ P \in \PC(X) \, | \, P \neq Q, b \notin P\} \  \cup \\
    &  \big\{ (P \setminus \{b\}) \cup \{x_1,a\} \ \big| \ b \in P\in \PC(X) \ \& \ (P \setminus \{b\})\cup \{\alpha\} \notin \PC(X) \text{ for } \alpha\in \{a,x_1\}  \big\}. 
    \end{align*}
    Part 1 and part 2 follow immediately from this description and \cref{lemma: RHS of an RPR is not contained in the centered PC}.

    We now prove part 3. Suppose there exists a relevant primitive collection $P$ in $X_\gamma$ which is not a primitive collection in $X$. It follows from the description of $\PC(X_\gamma)$ above that 
    $P = \{x_i, x_1, a\}$ for some $i\neq 1$ and $P=(P' \setminus \{b\}) \cup \{x_1,a\}$ for some $P'\in \PC(X)$.
    Then $P'$ must be one of the following: 
    \begin{itemize}
        \item[(1)] $P_1 =\{x_i, b\}$;
        \item[(2)]  $P_2 =\{x_i, x_1, b\}$; or 
        \item[(3)] $P_3 =\{x_i, a, b\}$. 
    \end{itemize}
    We show that neither (2) nor (3) can occur.
    Without loss of generality, suppose that $i=0$ and $P = \{x_0, x_1, a\} \in \PC(X_\gamma)$ comes from either $P_2=\{x_0,x_1, b\} \in \PC(X)$, or $P_3 = \{x_0, a, b\} \in \PC(X)$.
    In either case, $\langle b, x_0 \rangle$ is a cone in \(\Sigma_X\). The primitive collection $P$ gives a primitive relation on $X_\gamma$ 
    \[
    r_{X_\gamma}(P) \colon x_0+x_1 +a = \sum_j\mu_jv_j, \quad \mu_j>0,
    \] 
    where $\langle\overline{v}\rangle\in \Sigma_{X_\gamma}$ and $a,b, x_1,x_0 \notin \overline{v}$.
    By \cref{proposition: fan of a blowup}, $\langle \overline{v} \rangle$ is  also a cone in $\Sigma_{X}$. 
    On the other hand, substituting the primitive relation $x_1+a = b$ in $r_{X_\gamma}(P)$ gives the following relation:
    \[\sum_j\mu_jv_j = b+ x_0 ,\]
    with vectors on each side forming a cone in $X$, which is a contradiction.
    We conclude that $P=(P_1 \setminus \{b\}) \cup \{x_1,a\}$ with $P_1=\{x_0, b\} \in \PC(X)$.
\end{proof}

\begin{proposition} \label{prop: rpr after blowndown}
    Let $X$ be a proper toric manifold with a centered primitive relation $x_0+x_1+x_2=0$. 
    Assume that each generator of $\Sigma_X$ admits at most one opponent, and that the relation corresponding to any relevant primitive collection of $X$ is of type \ref{item: RPC x+a=b}, \ref{item: RPC x+x+a=b+c} or \ref{item: RPC x+x+a=b}.
    Let $X \to X_\gamma$ be the morphism corresponding to a contractible primitive relation of the form
    \begin{equation}
    \label{eq: prop blowdown induction: PC on Xgamma}\gamma = r_X(Q) \colon \, x_1+a=b.
    \end{equation}
     If $P$ is a relevant primitive collection of both $X$ and $X_\gamma$, then $r_X(P)=r_{X_\gamma}(P)$.
\end{proposition}

\begin{proof}
    Let $P$ be a relevant primitive collection of both $X$ and $X_\gamma$. 
    Let
    \[
    r_{X}(P) \colon \sum P = \beta b + \overline \eta \cdot \overline v
    \]
    be its corresponding primitive relation in $X$, where $\beta \geq 0$ and $\overline \eta$ is a vector of positive integers, while $\overline v$ is a tuple of vectors in $G(\Sigma_{X})$ not containing $b$ and spanning a cone in $\Sigma_{X}$.
    Since $\Sigma_{X}$ is obtained from $\Sigma_{X_\gamma}$ by the star subdivision of $\langle x_1, a \rangle$, we see that $\overline v$ still spans a cone in $\Sigma_{X_\gamma}$. 
    Therefore, it suffices to show that $\beta = 0$, because then \(\sum P = \overline \eta \cdot \overline v\) is evidently a primitive relation in \(X_\gamma\).
    
    Assume the contrary.
    By the assumptions on the relevant primitive relations of $X$, $r_X(P)$ has one of the following forms:
    \[x_i+c=b, \quad x_i+x_j+c = b + d , \quad x_i+x_j+c =  b 
    .\]
    
    Consider the first case, where we have $x_i+c = b$ on $X$. 
    Note that $i \not = 1$, otherwise $c = a$ and $\{x_1,a\}$ forms at the same time a cone and a primitive collection in $X_\gamma$, which is impossible. 
    By adding $x_i+c = b$, \eqref{eq: prop blowdown induction: PC on Xgamma} and $x_j=x_j$ for $j \not = 1, i$, we get $x_0+x_1+x_2+a+c=2b+x_j$, and the relation $x_0+x_1+x_2=0$ implies that $a+c=2b+x_j$.
    By \cref{remark: cannot have relation a+b=2c+d with a b forming a cone}, there is a primitive relation \(x_j + b = \bar d \) with \(|\bar d|>1\), which contradicts the assumption on the form of relevant primitive relations of $X$.

    We now treat the two remaining cases at once: assume that on $X$ we have a primitive relation $x_i+x_j +c= b+d$, with possibly $d = 0$. 
    In particular, $\langle b,d \rangle$ is a one- or two-dimensional cone in $\Sigma_{X}$.
    We first show that $x_i \neq x_1$. If $x_i = x_1$, then subtracting \eqref{eq: prop blowdown induction: PC on Xgamma} from this relation yields $x_j+c=a+d$. Here, $x_j$ and $c$ form a cone because $\{x_i,x_j,c \}\in \PC(X)$, while $d$ and $a$ form a cone since $a$ can have at most one opponent. This is a contradiction, and therefore the posited relation is $x_0 + x_2 + c = b + d$. 
    But now adding \eqref{eq: prop blowdown induction: PC on Xgamma} yields 
    $a+c = 2b+d$.
    Since the opponent of \(a\) is \(x_1\), we have that \(\{a,c\}\) spans a cone, while $\{b,d\}$ spans a cone from the above, and this leads to a contradiction.
\end{proof}

% ---
\section{Step 1: Construction of birational maps}
\label{sec:construction maps}
% ---

Let $X$ be a toric Fano manifold with $\dim(X)>2$ and $m(X) = 2$, 
and fix a centered primitive relation 
$$x_0 + x_1 + x_2 = 0.$$
As outlined in \cref{construction}, we wish to modify $X$ in two steps. 
First, we will perform a sequence of at most three smooth blowdowns $X\to X'$.
While $X'$ may cease to be Fano, we will show that $X'$ does not have any new relevant primitive relations, i.e., if $P$ is a relevant primitive collection of $X'$, then $P$ is also a relevant primitive collection of $X$, and their corresponding primitive relations are the same.
It will follow from our construction of $X\to X'$ that $X'$ does not have relevant primitive relations of type \ref{item: RPC x+a=b}.
As a consequence, $X'$ can only have relevant primitive relations of type \ref{item: RPC x+x+a=b+c}.
We will show in \cref{lemma: flipping loci do not intersect} that the flipping loci corresponding to these relations do not intersect, and so we can perform all flips at once, yielding a birational map $X'\dashrightarrow Y$.
While $Y$ may not be projective, it still carries the same centered primitive relation we fixed in $X$, and no relevant primitive relations.
As a result, by \cref{proposition: Pk bundle}, $Y$ contains a big open subset with a $\PP^2$-bundle structure.

% -
\subsection{Blowdowns}
\label{subsection: blow-downs of type 1}
% -

Let $X$ be a toric Fano manifold of minimal \(\PP\)-dimension $m(X)=2$.
Choose a centered primitive collection $\cpc$ with primitive relation
\begin{equation}
\label{eq: fixed CSPR}
    r(\cpc) \colon x_0 + x_1 + x_2 = 0 .
\end{equation}

The goal is to construct a finite sequence of smooth blowdowns whose composition $X \to X'$ yields a proper toric manifold $X'$ such that $\cpc$ is still centered and primitive in $X'$, and all relevant primitive collections in $X'$ are of type \ref{item: RPC x+x+a=b+c} and contractible.

\begin{proposition}
\label{proposition: doing all blow-downs}
Let $X$ be a toric Fano manifold with $m(X)=2$ and fix a centered primitive collection $\cpc$ as in (\ref{eq: fixed CSPR}).
Then there exists a toric morphism $X \to X'$ which is a composition of at most three blowdowns, where $X'$ is a proper toric manifold that satisfies the following properties:
\begin{itemize}
    \item $\cpc$ is a primitive collection of $X'$.
    \item 
    No new opponents appear on $X'$, i.e., if $\{u, v\} \in \PC(X')$ then $\{u, v\} \in \PC(X)$.
    \item
    If $P \in \RPC(X')$, then $P \in \RPC(X)$, $r_{X'} (P) = r_X (P)$ is of type \ref{item: RPC x+x+a=b+c} and contractible in $X'$.
\end{itemize}
\end{proposition}

\begin{proof}
If there are no relevant primitive relations of type \ref{item: RPC x+a=b}, then we set \(X \to X':=X\) to be the identity morphism.

From now on, we assume that there exists a relevant primitive relation of type \ref{item: RPC x+a=b}.
Without loss of generality, we may assume $P_1 = \{x_1, a\}$ is a primitive collection  of $X$,
and write \(r(P_1) : x_1 + a = b\) for the corresponding relevant primitive relation.
Let \(X \to X^1\) denote the divisorial contraction associated to \(r(P_1)\).

\paragraph{First case: $(X, \cpc)$ is exceptional, with exceptional decomposition \eqref{eq: exceptional decomposition}.}
By \cref{prop:PC_blowdown}, $P_0 \coloneqq \{x_0,c\} \in \PC(X^1)$, and $r_{X^1}(P_0) \colon x_0 + c= a$ is a contractible relation of $X^1$ by \cref{lem:contractible after blowdown}.
Denote by $f^1\colon X^1 \to X'$ the divisorial contraction induced by $r_{X^1}(P_0)$.
It follows immediately from \cref{prop:PC_blowdown} that $\cpc \in \PC(X')$.
By \cref{lem:P and opponents after blowdown}, no new opponents appear in \(X^1\) and $X'$, so any generator of \(\Sigma_{X^1}\) and $\Sigma_{X'}$ has at most one opponent.

We will now prove that $\RPC(X^1) \subseteq \RPC(X) \cup \big\{\{x_1,x_2,a\}\big\}$ and $\RPC(X') \subseteq \RPC(X)$.
The first inclusion follows from \cref{lem:P and opponents after blowdown} (\ref{item: new RPC on a contraction}).
Now suppose there exists a relevant primitive collection $R$ in $X'$ which is not a primitive collection in $X^1$.
Again by \cref{lem:P and opponents after blowdown} (\ref{item: new RPC on a contraction}), $R=\{x_i,x_0,c\}=(\{x_i,a\} \setminus \{a\}) \cup \{x_0,c\}$, where $R_1 \coloneqq \{x_i,a\} \in \PC(X^1)$ and $i \neq 0$. 
From the first inclusion, $\{x_i,a \} \in \PC(X)$, and by the uniqueness of opponents, we have necessarily $i=1$.
It follows that $\{x_1,a\} \in \PC(X^1)$ and at the same time $\langle x_1, a \rangle \in \Sigma_{X^1}$, which is impossible.

Finally, it follows from \cref{prop: rpr after blowndown} and \cref{lem:contractible after blowdown} applied to $X$ and $X^1$ that $r_{X'}(P)=r_X(P)$ for any $P \in \RPC(X')$, and that this relation is contractible in $X'$.
Furthermore, since \(G(\Sigma_{X'})\) contains neither \(a\) nor \(b\) by \cref{proposition: fan of a blowup}, there are no more relevant primitive relations of type \ref{item: RPC x+a=b} in $X'$.
Thus, by \cref{lemma:fewer_possibilities}, all relevant primitive relations in $X'$ are of type \ref{item: RPC x+x+a=b+c}. 

\paragraph{Second case: $(X, \cpc)$ is not exceptional.}
Let \(L\) be the number of relevant primitive relations of type \ref{item: RPC x+a=b} of $X$.
By \cref{remark: at most 3 of type x+a=b}, \(L\leq 3\).
Set $X^0 \coloneqq X$ and recall from the beginning of the proof that \(X^0 \to X^1\) is the divisorial contraction associated to \(r(P_1) : x_1 + a = b\).
By \cref{lem:P and opponents after blowdown}, 
\begin{enumerate}
    \item $\cpc \in \PC(X^1)$;
    \item no new opponents appear in $X^1$, i.e., if $\{u, v\} \in \PC(X^1)$ then $\{u, v\} \in \PC(X^{0})$, and any generator of $\Sigma_{X^1}$ has at most one opponent.
\end{enumerate}
We claim that moreover $\RPC(X^1) \subseteq \RPC(X^{0})$.
Suppose that there exists a relevant primitive collection $P$ of $X^{1}$ which is not a primitive collection of $X^0$.
By part \ref{item: new RPC on a contraction} of \cref{lem:P and opponents after blowdown}, $P=\{x_i,x_1,a\}=( \{x_i,b\} \setminus \{b\}) \cup \{x_1,a\}$ where $Q \coloneqq \{x_i,b\} \in \PC(X^0)$ and $i \neq 1$, say $Q=\{x_2,b\}$. Then the associated relation of $X^0$ is of the form $x_2+b=c$.
It follows that in $X^{0}$ we have the relations \eqref{eq: exceptional decomposition}, thus $(X,\cpc)$ is an exceptional case, which contradicts the assumption of the second case and proves the claim.
By \cref{prop: rpr after blowndown}, for any $P \in \RPC(X^1)$, we have $r_{X^1}(P)=r_{X^{0}}(P)$.
If this is a relevant primitive relation of type \ref{item: RPC x+a=b} or \ref{item: RPC x+x+a=b+c}, then it is contractible in $X^1$ by \cref{lem:contractible after blowdown}.

Let us proceed by induction.
Assume that for any $k=1,\dots,\ell$, where $\ell < L$, the morphism $X^{k-1} \to X^k$ is induced by a relevant primitive relation of type \ref{item: RPC x+a=b} of $X^{k-1}$, and that 
\begin{enumerate}
    \item $\cpc \in \PC(X^k)$;
    \item no new opponents appear in $X^k$, i.e., if $\{u, v\} \in \PC(X^k)$ then $\{u, v\} \in \PC(X^{k-1}) \cap \PC(X)$, and any generator of $\Sigma_{X^k}$ has at most one opponent;
    \item $\RPC(X^k) \subseteq \RPC(X^{k-1}) \subseteq \RPC(X)$;
    \item for any $P \in \RPC(X^k)$, we have $r_{X^k}(P)=r_{X^{k-1}}(P)=r_X(P)$ and, if this is a relevant primitive relation of type \ref{item: RPC x+a=b} or \ref{item: RPC x+x+a=b+c}, then it is contractible in $X^k$.
\end{enumerate}
Since $\ell < L$, by the induction hypothesis, there is still a relevant primitive relation of type \ref{item: RPC x+a=b}, call it $P' \in \RPC(X^\ell)$, and it is contractible in $X^\ell$. 
Denote by $X^\ell \to X^{\ell+1}$ the morphism induced by $r_{X^\ell}(P')$. 
The same argument used to treat the morphism \(X^0 \to X^1\) above shows that $X^{\ell+1}$ satisfies the properties 1--4 above.

Finally, it follows from properties 1--4 that, after performing the $L$ blowdowns induced by relevant primitive relations of type \ref{item: RPC x+a=b}, we obtain a proper toric manifold $X'$ with no more relevant primitive relations of type \ref{item: RPC x+a=b}.
By \cref{lemma:fewer_possibilities},  $X'$ does not have relevant primitive relations of type \ref{item: RPC x+x+a=b} either.
This concludes the proof.
\end{proof}

\begin{remark}
There are examples of exceptional cases having a unique centered primitive relation, illustrating the need to consider the exceptional case separately in the proof of \cref{proposition: doing all blow-downs}. This happens, for instance, for the varieties with Macaulay2 IDs (6,~276), (6,~333), and (6,~338).
\end{remark}

% -
\subsection{Flips}
% -

Let $X$ be a toric Fano manifold with $m(X)=2$ and a centered primitive relation $x_0+x_1+x_2=0$.
By \cref{proposition: doing all blow-downs}, there is a toric birational morphism $X \to X'$ such that $x_0+x_1+x_2=0$ is still a primitive relation in $X'$, no new opponents appear in $X'$, $\RPR (X') \subseteq \RPR (X)$, and the relevant primitive relations are all of type \ref{item: RPC x+x+a=b+c} and contractible in $X'$.
In this subsection, we construct a toric birational map $X'\dashrightarrow Y$ such that $Y$ is smooth and proper, it still has $x_0+x_1+x_2=0$ as a primitive relation, but no relevant primitive relation.

We will prove auxiliary lemmas in a slightly more general setting.
\begin{assumption}
\label{assumption: X'}
    Let $X'$ be a proper toric manifold such that 
    \begin{enumerate}
        \item $X'$ has a centered primitive relation $r(\cpc) \colon x_0+x_1+x_2=0$;
        \item Any generator of $\Sigma_{X'}$ has at most one opponent;
        \item All relevant primitive relations of $X'$ are contractible and of type \ref{item: RPC x+x+a=b+c}, i.e., of the form $x_i + x_j + a = b + c$.
    \end{enumerate}
\end{assumption}

\begin{setup}
\label{notation: flipping loci}
\label{rmk:how the cones differ in a flip}
Under \cref{assumption: X'}, let us denote the relevant primitive collections of $X'$ by $\RPC(X') = \{ Q_1, \dots, Q_N \}$.
Fix a relevant primitive collection
\[
Q_i=\{x_k,x_\ell,a\},
\quad
r(Q_i):x_k+x_\ell+a=b+c,
\]
and set \(R_i=\{b,c\}\). Since \(r(Q_i)\) is contractible, we can perform the flip described in \cref{setup_flips}:
    \[\begin{tikzcd}
    	& {\tilde X} \\
    	X' && X'' \, ,
    	\arrow["{b + c = z}"', from=1-2, to=2-1]
    	\arrow["{x_k + x_\ell + a = z}", from=1-2, to=2-3]
    	\arrow[dashed, from=2-1, to=2-3]
    \end{tikzcd}\]
    with $X''$ smooth and proper. 
    We denote by $C_i := V_{X'} (b,c) \subset X'$ the center of the flip. 
\end{setup}

\begin{lemma}
\label{lemma: flipping loci do not intersect}
    Let $X'$ satisfy \cref{assumption: X'}, and follow the notation introduced in \cref{notation: flipping loci}.
    Then $C_i \cap C_j = \varnothing$ for distinct $i$ and $j$, i.e., the centers of the flips are pairwise disjoint.
\end{lemma}

\begin{proof}
    Let us write $r(Q_i) \colon x_\ell + x_k + a = b + c$ and
    $r(Q_j) \colon x_\ell + x_{k'} + a' = b' + c'$ for the relations with flipping centers $C_i$ and $C_j$, respectively.
    In order to prove the claim, we will argue by contradiction and assume that \(C_i \cap C_j \neq \varnothing\), or equivalently, that $\langle b, c, b', c' \rangle$ is a cone in $\Sigma_{X'}$.

    Substitution yields the relation $x_k + a + b' + c' = x_{k'} + a' + b + c$.
    Apply \cref{prop: contractible} to the contractible primitive relation $r(Q_i)$ and the cone 
    \(\tau = \langle \{b',c'\}\setminus \{a,b,c\} \rangle\).
    Together with the assumption $\langle b, c, b', c' \rangle \in \Sigma_{X'}$, this implies that $\langle x_k, a, b, c, b', c' \rangle \in \Sigma_{X'}$.
    Similarly, contractibility of $r(Q_j)$ implies that $\langle x_k, x_{k'}, a, b, c, a', b', c' \rangle \in \Sigma_{X'}$.
    We conclude that vectors on both sides of the relation $x_k + a + b' + c' = x_{k'} + a' + b + c$ span cones in $\Sigma_{X'}$, which is only possible if \(x_k = x_{k'}\) and \(\{a,b',c'\} = \{a',b,c\} \).
    But \(a\neq a'\) because \(Q_i \neq Q_j\), and \(a\neq b,c\) by \cref{remark: PC does not intersect focus}, which leads to a contradiction.
\end{proof}

\begin{remark}
\label{notation: all flips at once}
\label{remark: all flips at once}
Let the notation be as in \cref{notation: flipping loci}.
Recall from \cref{setup_flips}
that the modification of the fan $\Sigma_{X'}$ under the flip of $r(Q_i) \colon x_\ell + x_k + a = b + c$ only replaces cones $\sigma\in \Sigma_{X'}$ containing $\langle R_i\rangle=\langle b,c\rangle$. 
By \cref{lemma: flipping loci do not intersect}, under \cref{assumption: X'}, the centers of the flips \(C_1,\ldots,C_N\) are pairwise disjoint. In terms of the fan, this means that no cone of \(\Sigma_{X'}\) contains simultaneously two cones
\(\langle R_i\rangle\) and \(\langle R_j\rangle\) with \(i\ne j\).
So, the sets of cones removed by each flip are disjoint. 
Moreover, no new cone that appears by flipping $r(Q_i)$ contains \(\langle R_j\rangle\) for \(j\ne i\).
Indeed, every new cone is a face of a cone
\(\langle Q_i,R_i\setminus\{v\},\sigma\rangle\) for some $\sigma\in \Sigma_{X'}$ such that $G(\sigma)\cap(Q_i\cup R_i)=\varnothing$ and some $v\in R_i$. 
If this cone contained \(\langle R_j\rangle\) for \(j\ne i\), then the old cone
\(\langle Q_i\setminus\{x_k\},R_i,\sigma\rangle\) would contain both \(\langle R_i\rangle\) and \(\langle R_j\rangle\), contradicting \cref{lemma: flipping loci do not intersect}.

Consequently, the fan modifications described in \cref{setup_flips} do not interfere and we can perform all the flips in any order, with the same resulting fan.
This yields a proper toric manifold \(Y\) and a birational map
\(X'\dashrightarrow Y\). 
We set $U' := X' \setminus \bigcup_{i=1}^N C_i$
and denote by $Z_i \subset Y$ the flipped locus corresponding to $r(Q_i)$. Then we have the following identifications:
    \[
    X' = U' \sqcup \bigsqcup_{i=1}^N C_i
    \quad \text{ and } \quad
    Y = U' \sqcup \bigsqcup_{i=1}^N Z_i,
    \]
    with $\codim_{X'} C_i = 2$ and $\codim_Y Z_i = 3$.
\end{remark}

\begin{proposition}
\label{lemma: no new RPR after all flips}
    Let $X'$ satisfy \cref{assumption: X'}, and let $f \colon X' \dashrightarrow Y$ denote the flipping of all $\RPR(X')$ as in \cref{notation: all flips at once}.
    Then $Y$ is a proper toric manifold such that:
    \begin{enumerate}
        \item $G(\Sigma_Y) = G(\Sigma_{X'})$;
        \item
        \label{item: <x,x,a> is a cone of Y}
        For any distinct $x_i, x_j \in \cpc$ and $a \in G(\Sigma_Y) \setminus \cpc$, the cone $\langle x_i, x_j, a \rangle$ is in the fan $\Sigma_Y$;
        \item
        \label{item: x+x+x=0 remains a PC on Y}
        The centered primitive collection $\cpc$ of $X'$ remains a primitive collection of $Y$.
    \end{enumerate}
    In other words, $\RPC(Y) = \varnothing$.
\end{proposition}

\begin{proof}
    Recall the notation from \cref{remark: all flips at once}: \(U' = X' \setminus \bigsqcup C_i = Y \setminus \bigsqcup Z_i\).
    Since $U'$ is an open toric submanifold of both $X'$ and $Y$, with complements of codimension $\geq 2$, we have $G(\Sigma_Y) = G(\Sigma_{U'}) = G(\Sigma_{X'})$.

    We now turn to proving Part \ref{item: <x,x,a> is a cone of Y}.
    Recall \cref{notation: all flips at once}, and set $n := \dim X' = \dim Y$.
    Then the claim in part \ref{item: <x,x,a> is a cone of Y} is equivalent to showing that $V_Y(x_i) \cap V_Y(x_j) \cap V_Y(a) \neq \varnothing$.
    But this set contains
    \begin{align*}
    V_Y(x_i) \cap V_Y(x_j) \cap V_Y(a) \cap U' = V_{X'}(x_i) \cap V_{X'}(x_j) \cap V_{X'}(a) \cap U'.
    \end{align*}
    \\ \noindent \textbf{Case 1: $\{x_i, x_j, a\}$ is a primitive collection of $X'$.}
    Then $\{x_i, x_j, a\} = Q_\ell \in \RPC(X')$ for some $\ell$, which means that $V_Y(x_i) \cap V_Y(x_j) \cap V_Y(a) = V_Y(x_i, x_j, a) = Z_\ell \neq \varnothing$.
    \\ \noindent \textbf{Case 2: $\{x_i, x_j, a\}$ is not a primitive collection of $X'$.}
    All proper subsets of $\{x_i, x_j, a\}$ span cones by \cref{assumption: X'}, hence $\{x_i, x_j, a\}$ must span a cone of $X'$ in order to avoid being a primitive collection.
    We will now show that $V_{X'}(x_i, x_j, a)$ cannot be a subset of any $C_\ell = V_{X'} (b_\ell, c_\ell)$.
    If it were, we would have $\{ b_\ell, c_\ell \} \subset \{x_i, x_j, a\}$, which is impossible because $b_\ell, c_\ell \notin P = \{ x_0 , x_1 , x_2 \}$ by \cref{lemma: RHS of an RPR is not contained in the centered PC}.
    We then get that $V_{X'}(x_i, x_j, a) \cap U' = V_{X'}(x_i, x_j, a) \setminus \bigsqcup_\ell C_\ell \neq \varnothing$, as desired.

    For Part \ref{item: x+x+x=0 remains a PC on Y}, notice that every pair $\{x_i,x_j\}$ spans a cone of $Y$ by part \ref{item: <x,x,a> is a cone of Y}, but $\cpc = \{x_0, x_1, x_2\}$ cannot span a cone, because $x_0 + x_1 + x_2 = 0$.
\end{proof}

\begin{theorem}
\label{theorem: construct blowdowns and flips}
    Let $X$ be a toric Fano manifold of dimension $n > 2$ and $m(X) = 2$.
    Let $\cpc$ be a centered primitive collection of order 3 of $X$.
    Then there exist toric birational maps $X \to X' \dashrightarrow Y$ such that:
    \begin{itemize}
        \item $\cpc$ is a primitive collection of both $X'$ and $Y$;
        \item $X \to X'$ is a composition of at most 3 blowdowns, and $\RPR(X')$ only contains relations of type \ref{item: RPC x+x+a=b+c};
        \item $X' \dashrightarrow Y$ is a composition of toric flips with disjoint centers;
        \item $Y$ contains an open subset $U \subset Y$ with $\codim (Y \setminus U) \geq 2$ on which $r(\cpc) \colon x_0+x_1+x_2 = 0$ induces a $\PP^2$-bundle $U \to W$.
    \end{itemize}
\end{theorem}

\begin{proof}
    This is a combination of \cref{proposition: Pk bundle}, \cref{proposition: doing all blow-downs} and \cref{lemma: no new RPR after all flips}.
\end{proof}

% ---
\section{Step 2: Existence of a required surface}
\label{section:step2}
% ---

In this section, we carry out Step~2 of the strategy outlined in the introduction.
We start by establishing the following general obstruction to being $2$-Fano.  

\begin{theorem}
\label{thm_existence_of_surface}
Let $Y$ be a complex manifold admitting a $\PP^m$-bundle structure $\pi \colon U\to W$ ($m\geq 1$) on a dense open subset $U\subset Y$, and assume that $W$ contains a complete rational curve $C\subset W$. 
Write $\mathbb P^1 \to C$ for the normalization morphism, set $U_C:=\pi^{-1}(C)$ and consider the fiber product diagram:
\[\begin{tikzcd}
	U_{\mathbb P^1} & {U_{C}} & U & Y \\
	{\mathbb P^1} & {C} & W
	\arrow["{\pi'}"', from=1-1, to=2-1]
	\arrow["\nu", curve={height=-15pt}, from=1-1, to=1-4]
	\arrow["\gamma"', curve={height=15pt}, from=2-1, to=2-3]
	\arrow["\pi", from=1-3, to=2-3]
	\arrow[from=2-1, to=2-2]
	\arrow[hook, from=2-2, to=2-3]
	\arrow[from=1-1, to=1-2]
	\arrow[hook, from=1-2, to=1-3]
	\arrow[from=1-2, to=2-2]
    \arrow[hook, from=1-3, to=1-4]
\end{tikzcd}\]
Then  
\begin{enumerate}
\item \label{item: restriction of ch2 from Y to P} $\nu^* \ch_2 (Y) = \ch_2 (U_{\mathbb P^1})$; and
\item \label{item: existence of surface with nonpositive intersection}
for a suitable choice of $\PP^1$-subbundle $S_{\PP^1}\subset U_{\PP^1}$ over $\PP^1$, the surface 
\[S:=\nu_*\big(S_{\PP^1}\big)\subset U\subset Y\] satisfies 
\[ S \cdot \ch_2(Y) \leq 0 \, .\]
\end{enumerate}
\end{theorem}

\begin{proof}
To prove Part \ref{item: restriction of ch2 from Y to P}, consider the relative tangent sequences:
\[
0\longrightarrow \nu^*T_{U/W}
\longrightarrow \nu^*T_Y
\longrightarrow (\pi')^*\gamma^*T_W
\longrightarrow0\, ,
\]
\[
0\longrightarrow T_{U_{\mathbb P^1}/\mathbb P^1}
\longrightarrow T_{U_{\mathbb P^1}}
\longrightarrow (\pi')^*T_{\mathbb P^1}
\longrightarrow0\, ,
\]
and note that 
\[
T_{U_{\mathbb P^1}/\mathbb P^1}
\cong \nu^*T_{U/W}.
\]
By additivity of Chern characters, we get that 
\[
\nu^*\ch_2(T_Y)-\ch_2(T_{U_{\mathbb P^1}})=(\pi')^*\big(\ch_2(\gamma^*T_W)-\ch_2(T_{\mathbb P^1})\big).
\]
The right-hand side vanishes because both terms in parentheses are codimension-two classes on $\mathbb P^1$, and Part \ref{item: restriction of ch2 from Y to P} follows. 

    For Part \ref{item: existence of surface with nonpositive intersection}, let $a_i\in \ZZ$ be such that  
    \[
    U_{\mathbb P^1} \cong P_{\mathbb P^1}\big(\cO_{\mathbb P^1}(a_0) \oplus \dots \oplus \cO_{\mathbb P^1}(a_m)\big)
    \]
    and $a_0 \leq a_1 \leq \cdots \leq a_m$, and set $A := \sum_{i=0}^m a_i$.
    
    The relative tangent sequence, the relative Euler sequence, and additivity of Chern characters imply that
    \[
    \ch_2 (U_{\PP^1}) = \frac{m+1}{2} \xi^2 - AF \cdot \xi ,
    \]
    where $\xi$ denotes the tautological relative hyperplane class on $U_{\PP^1}$ and $F=(\pi')^*[\mathrm{pt}]$.
    
    Let $S_{\PP^1} \cong \PP_{\PP^1} ( \cO_{\PP^1}(a_0) \oplus \cO_{\PP^1}(a_1) )\subset U_{\PP^1}$ be the $\PP^1$-subbundle corresponding to the projection 
    \[
        \cO_{\PP^1}(a_0) \oplus \cO_{\PP^1}(a_1) \oplus \dots \oplus \cO_{\PP^1}(a_m)\twoheadrightarrow 
        \cO_{\PP^1}(a_0) \oplus \cO_{\PP^1}(a_1)
    \]    
    (we use Grothendieck's convention for projective bundles).
    Recall the projective bundle formula on the surface $S_{\PP^1}$:
    \begin{align*}
        \xi_{|S_{\PP^1}}^2 &=
        - \cc_1 \big((\pi')^*\big(\cO_{\PP^1}(-a_0) \oplus \cO_{\PP^1}(-a_1)\big)\big) \cdot \xi_{|S_{\PP^1}} - \cc_2 \big((\pi')^* \big(\cO_{\PP^1}(-a_0) \oplus \cO_{\PP^1}(-a_1)\big)\big)
        \\ &=
        (a_0+a_1) F \cdot \xi_{|S_{\PP^1}} - 0 =
        a_0 + a_1.
    \end{align*}
    Substitution yields
    \begin{align*}
        S_{\PP^1} \cdot \ch_2 (U_{\PP^1}) =
        \frac{m+1}{2} (\xi_{|S_{\PP^1}})^2 - A =
        \frac{m-1}{2} a_0 + \frac{m-1}{2} a_1 - (a_2 + \cdots + a_m)
        \leq 0.
    \end{align*}
    For $S = \nu_* (S_{\PP^1})$, the projection formula, Part \ref{item: restriction of ch2 from Y to P} and the last inequality yield:
    \[
    S \cdot \ch_2(Y) = S_{\PP^1} \cdot \nu^* \ch_2 (Y)
    = S_{\PP^1} \cdot \ch_2 (U_{\mathbb P^1}) \leq 0.
    \qedhere
    \]
\end{proof}

We will apply \cref{thm_existence_of_surface} to a proper toric manifold $Y$ admitting a $\PP^m$-bundle structure on the complement of a closed subset of codimension at least $2$.
In this case, the existence of the complete rational curve $C\subset W$ is guaranteed by the following lemma. We need to choose this curve carefully in order to perform the necessary Chern character computations in the next section. 

\begin{lemma}
\label{lemma:existence of curve}
Let $Y$ be a proper toric manifold. Suppose that there exists a closed subset $Z\subset Y$ of codimension at least $2$, a positive dimensional toric manifold $W$, and a toric $\PP^m$-bundle $Y\setminus Z\to W$. 
Then there is a rational curve $C\subset W$ with the property that $C$ intersects transversely the torus-invariant divisors of $W$ and is disjoint from the intersection of any two distinct torus-invariant divisors of $W$. 
Moreover, if $\dim(W)\geq 3$, we can take $C$ to be smooth.
\end{lemma}
    
\begin{proof}
Set $U := Y \setminus Z$ and denote by $\pi \colon U \to W$ the toric $\mathbb P^m$-bundle.

By \cite[Corollary 1.1]{ChenShokurov2011}, there is a very free rational curve $C'\subset Y$. Since $Z$ has codimension $\geq 2$ in $Y$, a general deformation of $C'$ does not intersect $Z$ by \cite[Proposition II.3.7]{kollar96}. So we may assume that $C'\subset U$. For the same reason, we may assume that $C'$ is disjoint from the intersection of any two distinct torus-invariant divisors of $U$. 
By \cite[Proposition 2.9]{BLRT22}, $C'$ can also be taken to intersect transversely the torus-invariant divisors of $U$.

Since $C'$ is very free, it cannot be contained in a fiber of $\pi$, and we set $C=\pi(C')\subset W$. 
The pullbacks via $\pi$ of the torus-invariant divisors of $W$ are torus-invariant divisors of $U$. 
Therefore, $C$ intersects transversely the torus-invariant divisors of $W$ and is disjoint from the intersection of any two distinct torus-invariant divisors of $W$. 

From the construction of $C\subset W$, we see that a general deformation of the normalization morphism $\mathbb P^1 \to C\hookrightarrow W$ is very free by \cite[Corollary II.3.10.1]{kollar96}.
It follows from \cite[Theorem II.3.14]{kollar96} that we can take $\mathbb P^1 \to C$ to be an embedding if $\dim(W)\geq 3$.
\end{proof}

\begin{corollary}
\label{corollary: existence of surface with nonpositive intersection}
Let $Y$ be a proper toric manifold with a centered primitive collection $\cpc =\{x_0,\dots,x_m\}$ of order $m+1$, with $1 \leq m < \dim Y$.
Suppose that $V(\mathcal{E}_{\cpc})$ has codimension $\geq 2$ in $Y$, set $U := Y \setminus V(\mathcal E_{\cpc})$ and denote by $\pi \colon U \to W$ the $\mathbb P^m$-bundle induced by $r(\cpc)$. 
Let $C\subset W$ be a rational curve as in \cref{lemma:existence of curve} and set $U_C:=\pi^{-1}(C)$.
Then,
for a suitable choice of distinct indices $i,j\in\{0,\dots,m\}$, the surface 
\[S:=U_{C} \cap \bigcap_{k\neq i,j} V(x_k)\subset U\subset Y\] satisfies 
\[ S \cdot \ch_2(Y) \leq 0 \, .\]
\end{corollary}

\begin{proof}
Since $U \to W$ is a toric projective bundle, it can be written as $U \cong \mathbb P_W (\cU)$ for some split vector bundle $\cU = \cL_0 \oplus \cdots \oplus \cL_m$ on $W$. 
We index the $x_i$ so that $V(x_i)$ is identified with the divisor $\mathbb P_W(\cU / \cL_i)\subset U$ corresponding to the  projection $\cU \to \cU / \cL_i$.

Let the notation be as in \cref{thm_existence_of_surface} and its proof, and let $a_i\in \ZZ$ be such that $\gamma^*\cL_i\cong \cO_{\PP^1}(a_i)$. After relabeling the $x_i$'s if necessary, we may assume that $a_0 \leq a_1 \leq \cdots \leq a_m$. Recall from the proof of \cref{thm_existence_of_surface} that the surface 
\[
S_{\PP^1} \cong \PP_{\PP^1} ( \cO_{\PP^1}(a_0) \oplus \cO_{\PP^1}(a_1) )\subset \PP_{\mathbb P^1}\big(\cO_{\mathbb P^1}(a_0) \oplus \dots \oplus \cO_{\mathbb P^1}(a_m)\big) \cong U_{\PP^1}
\]
satisfies $S_{\PP^1} \cdot \ch_2 (U_{\mathbb P^1}) \leq 0$.
With the chosen indexing, $\nu_* (S_{\PP^1})=U_{C} \cap V(x_2) \cap \dots \cap V(x_m)$, and the 
result follows from \cref{thm_existence_of_surface}.
\end{proof}

% ---
\section{Step 3: Chern character computations}
\label{section: Chern character computations}
% ---

Let $X$ be a toric Fano manifold with $m(X)=2$. If $\dim(X)=n\leq 4$, then $m(X)=2\geq n-2$, and there is a complete classification of such toric Fano manifolds, which in particular gives that $X$ is not $2$-Fano except if $X=\PP^n$ with \(n\geq 2\) (\cite[Theorem 1.4 and Corollary 1.5]{team2023}).
So, from now on, we assume that $\dim(X) \geq 5$.
\begin{setup}\label{setup:loci}
Let $X$ be a toric Fano manifold with $m(X)=2$ and  $\dim(X) \geq 5$. Fix a centered primitive collection $\cpc = \{x_0,x_1,x_2\}$ of $X$, with centered primitive relation $r(\cpc)\colon x_0+x_1+x_2=0$. 
Consider a sequence of toric birational maps as constructed in \cref{sec:construction maps}:
\begin{center}
\begin{tikzcd}
    X\coloneqq X^0 \arrow[r, "f_0"] & \ \cdots \ \arrow[r, "f_{L-1}"] &X^L \arrow[r,dashed, "f_L"] & X^{L+1}  \arrow[r,dashed, "f_{L+1}"] & \ \cdots \ \arrow[r,dashed, "f_{k-1}"] &X^k \eqqcolon Y,
\end{tikzcd}    
\end{center}
so that $r(\cpc)$ is still a centered primitive relation of $Y$, and $V_Y(\mathcal{E}_{\cpc})$ has codimension $\geq 2$ in $Y$ (see \cref{theorem: construct blowdowns and flips}).
In particular,
\begin{itemize}
    \item for $0\leq\ell<L$, each $f_\ell\colon X^\ell\to X^{\ell +1}$ is the blowdown induced by a relevant primitive relation of $X$ of the form $x_i+a_i=b_i$. It contracts the divisor $V_{X^\ell}(b_i)\subset X^\ell$ onto $V_{X^{\ell+1}}(x_i,a_i)\subset X^{\ell+1}$;

    \item for $L\leq \ell<k$, each $f_\ell\colon X^\ell\dashrightarrow X^{\ell +1}$ is the flip corresponding to a relevant primitive relation of the form $x_i + x_j + a_{ij} = b_{ij} + c_{ij}$. It flips $C_\ell\coloneqq V_{X^\ell}(b_{ij},c_{ij})\subset X^\ell$ to $Z_\ell\coloneqq V_{X^{\ell+1}}(x_i,x_j,a_{ij})\subset X^{\ell+1}$. 
\end{itemize}
We abuse notation and denote by the same symbol $Z_\ell$ the strict transform of $Z_\ell\subset X^{\ell+1}$ on any $X^i$ for $i>\ell$, and similarly for the strict transform of $C_\ell\subset X^\ell$ on any $X^i$ for $i\leq\ell$.
\smallskip

By \cref{proposition: doing all blow-downs}, the construction of the sequence of blowdowns $X^0 \to X^L$ depends on whether or not $(X,\cpc)$ is an exceptional case  (\cref{def:exceptional_case}).
If $(X,\cpc)$ is an exceptional case, 
then $L=2$ and, for some choice of $i,j\in \{0,1,2\}$, the blowdowns $f_0 \colon X^0 \to X^1$ and $f_1 \colon X^1 \to X^L$ are associated to relevant primitive relations of the form $x_i+a=b$ and $x_j+c=a$, respectively. The morphism $f_0$ contracts $V_{X^0}(b)\subset X^0$ onto $V_{X^1}(x_i,a)\subset X^1$, while $f_1$ contracts $V_{X^1}(a)\subset X^1$ onto $V_{X^2}(x_j,c)\subset X^2$, hence the composed morphism maps $V_{X^0}(b)\subset X^0$ onto $V_{X^2}(x_i,x_j,c)\subset X^2$:
\[\begin{tikzcd}[row sep=small]
	{X^0} && {X^1} && {X^L} \\
	{V_{X^0}(a)} && {V_{X^1}(a)} && {V_{X^L}(x_j,c)} \\
	{V_{X^0}(b)} && {V_{X^1}(x_i,a)} && {V_{X^L}(x_i,x_j,c).}
	\arrow["{x_i+a=b}", from=1-1, to=1-3]
	\arrow["{x_j+c=a}", from=1-3, to=1-5]
	\arrow[from=2-1, to=2-3]
	\arrow[from=2-3, to=2-5]
	\arrow[from=3-1, to=3-3]
	\arrow[hook, from=3-3, to=2-3]
	\arrow[from=3-3, to=3-5]
	\arrow[hook, from=3-5, to=2-5]
\end{tikzcd}\]
If $(X,\cpc)$ is not an exceptional case, then the number of relevant primitive relations of type $x_i+a_i=b_i$ of $X$ is \(L\leq 3\) by \cref{remark: at most 3 of type x+a=b}, and the $a_i$'s and $b_i$'s are all pairwise distinct.
\smallskip

We make some observations which will be relevant for the computations that follow. 
\begin{enumerate}
    \item The image in $X^L$ of the exceptional locus of $X^0\to X^L$ is a union of subsets of the form $V_{X^L}(x_i,a_i)$, and these are not contained in the centers of the flips. Therefore, their strict transforms in $Y$ are well-defined and also of the form $V_{Y}(x_i,a_i)$. 

    \item Since the centers of the flips are pairwise disjoint in $X^L$, the loci $C_i$'s and $Z_j$'s are pairwise disjoint in every model $X^\ell$ where they are defined.
    
    \item On $X$, there cannot exist simultaneously relevant primitive relations of the form $x_i+a=b$ and $x_i+x_j+a=b'+c'$ (with the same vectors $x_i$ and $a$ appearing on the left hand side) by \cref{remark: when a may repeat}.
\end{enumerate}
\smallskip 

By construction, the set $V_Y(\mathcal{E}_{S_x})\subset Y$ has codimension $\geq 2$. Therefore,
by \cref{proposition: Pk bundle}, there is a toric manifold $W$, with $\dim(W)\geq 3$, and a $\mathbb{P}^2$-bundle structure $\pi \colon U \rightarrow W$, where $U=Y \setminus V_Y(\mathcal{E}_{S_x})$.
By \cref{lemma:existence of curve}, there is a smooth rational curve $C\subset W$ with the property that $C$ intersects transversely the torus-invariant divisors of $W$ and is disjoint from the intersection of any two distinct torus-invariant divisors of $W$.
Set $U_{C}:=\pi^{-1}(C)\subset U$.
By \cref{thm_existence_of_surface}, after a possible relabeling of the $x_i$'s, the surface $S \coloneqq U_{C} \cap V(x_2) \subset Y$ satisfies $ \ch_2(Y)\cdot S \leq 0$. 
We set $S^{k}\coloneqq S \subset Y$, and, for $0\leq \ell<k$, we let $S^{\ell} \subset X^{\ell}$ be the strict transform of $S$. 
\end{setup}

\smallskip

Our goal in this section is to compare Chern characters and show that $\ch_2(X^0)\cdot S^0 \leq \ch_2(Y)\cdot S \leq 0$. For this purpose, we analyze the intersection of each $S^\ell$ with the loci $V_{X^\ell}(x_i,a)$ and $Z_q\subset X^\ell$, $q<\ell$.
For any $a\in G(\Sigma_Y)\setminus S_x$, the divisor $V_Y(a)\cap U$ is the pullback via $\pi$ of a torus-invariant divisor of $W$. We abuse notation and denote this divisor by $V_W(a)\subset W$.

\begin{lemma} \label{lemma:intersections}
    Let the notation and assumptions be as in \cref{setup:loci}. Then the following holds.  
    \begin{enumerate}
        \item For every $0\leq \ell \leq k$, the composed map $\pi_\ell\colon S^\ell\dashrightarrow C$ is a morphism whose general fiber is isomorphic to $\PP^1$.
        \item For every $q<L$, consider the blowdown $f_q\colon X^q\to X^{q +1}$ associated to a relevant primitive relation of the form $x_i+a=b$, with $a\in G(\Sigma_Y)$.  If $\ell\geq q+1$, the intersection $V_{X^\ell}(x_i,a)\cap S^\ell$ of the image of the center of $f_q$ in $X^\ell$ with $S^\ell$ is:
        \begin{itemize}
            \item the union of finitely many smooth fibers of $\pi_\ell\colon S^\ell\to C$, one over each point of the intersection $C\cap V_W(a)$, if $i=2$;
            \item the union of finitely many smooth points of fibers of $\pi_\ell\colon S^\ell\to C$, one over each point of the intersection $C\cap V_W(a)$, if $i\neq2$. 
        \end{itemize}
        \item Suppose that $(X,S_x)$ is an exceptional case, hence $L=2$, and the blowdowns $f_0 \colon X^0 \to X^1$ and $f_1 \colon X^1 \to X^L$ are associated to relevant primitive relations of the form $x_i+a=b$ and $x_j+c=a$, respectively.
        
        For $\ell\geq 2$, the intersection $V_{X^\ell}(x_i,x_j,c)\cap S^\ell$ of the image of the center of $f_0$ in $X^\ell$ with $S^\ell$ is:
        \begin{itemize}
            \item the union of finitely many smooth points of fibers of $\pi_\ell\colon S^\ell\to C$, one over each point of the intersection $C\cap V_W(c)$, if $i=2$ or $j=2$;
            \item $\varnothing$, otherwise. 
        \end{itemize}

        The intersection $V_{X^1}(x_i,a)\cap S^1$ of the image of the center of $f_0$ in $X^1$ with $S^1$ is:
        \begin{itemize}
            \item the union of finitely many smooth curves of numerical class $x_j+c=a$, contained in fibers of $\pi_1\colon S^1\to C$, one over each point of the intersection $C\cap V_W(c)$, if $i=2$;
            \item the union of finitely many smooth points of fibers of $\pi_1\colon S^1\to C$, one over each point of the intersection $C\cap V_W(c)$, if $j=2$;
            \item $\varnothing$, otherwise. 
        \end{itemize}
        \item For every $q\geq L$, consider  the flip $f_q\colon X^q \dashrightarrow X^{q+1}$ associated to the relevant primitive relation $x_i+x_j+a=b+c$. If $\ell\geq q+1$, the intersection $Z_q\cap S^\ell$ of the flipped locus of $f_q$ with $S^\ell$ is:
        \begin{itemize}
            \item the union of finitely many smooth points of fibers of $\pi_\ell\colon S^\ell\to C$, one over each point of the intersection $C\cap V_W(a)$, if $i=2$ or $j=2$;
            \item $\varnothing$, otherwise. 
        \end{itemize}       
    \end{enumerate}
\end{lemma}

\begin{proof}
We prove the result by reverse induction on $\ell$. 
For $\ell=k$, by construction, $S=S^k$ is a smooth $\PP^1$-bundle over $C$ whose fibers have numerical class $x_0+x_1+x_2=0$ in $Y$.
The images in $Y=X^k$ of the centers of the blowdowns and the flipped loci are of the form $V_Y(x_i,a)$ and $V_Y(x_i,x_j,a)$. We analyze the intersection of $S$ with these loci:
\[
S \cap V_Y(x_i,a) \  = \
   \begin{cases}
       \text{union of the fibers over } C\cap V_W(a) & \text{if }i=2\\
       \text{union of points, one in each fiber over } C\cap V_W(a) & \text{if }i\neq 2  ,
   \end{cases}
\]
and 
\[
S \cap V_Y(x_i,x_j,a) = 
   \begin{cases}
       \text{union of points, one in each fiber over } C\cap V_W(a) & \text{if }i=2 \text{ or } j=2\\
       \varnothing & \text{otherwise} .
   \end{cases}
\]

Suppose the result holds for $\ell\geq m+1$, and consider the birational map $f_m\colon X^m \dashrightarrow X^{m+1}$.
\smallskip

First we consider the case $m\geq L$, when $f_m$ is the flip associated to a relevant primitive relation of the form $x_i+x_j+a=b+c$.
It flips $C_m= V_{X^m}(b,c)\subset X^m$ to $Z_m= V_{X^{m+1}}(x_i,x_j,a)\subset X^{m+1}$, and is an isomorphism elsewhere. 
The restriction ${f_m}_{|S^m}\colon S^m\to S^{m+1}$ is the blowup of $S^{m+1}$ at the finitely many points in $V_{X^{m+1}}(x_i,x_j,a)\cap S^{m+1}$, and hence the composed map $\pi_m\colon S^m\to C$ is a morphism whose general fiber is isomorphic to $\PP^1$. (Note that, if $i,j\neq 2$, then $V_{X^{m+1}}(x_i,x_j,a)\cap S^{m+1}=\varnothing$ and ${f_m}_{|S^m}\colon S^m\to S^{m+1}$ is an isomorphism.)
The union of the exceptional curves of the blowup ${f_m}_{|S^m}\colon S^m\to S^{m+1}$ is precisely $V_{X^m}(b,c)\cap S^m$.
By observation (2) made in \cref{setup:loci}, this locus is disjoint from the other flipped loci. In particular, 
\begin{enumerate}[label=(\MakeUppercase{\roman*})]
    \item
    For $L\leq q<m$, $Z_q\cap S^m=(f_m)^{-1}(Z_q\cap S^{m+1})$, and the desired description of the intersection $Z_q\cap S^m$ follows from the induction hypothesis.
    \item
    For $0\leq q<L$, let $f_q:X^q\to X^{q+1}$ be the blowdown associated to a relevant primitive relation of the form $x_{i'}+a'=b'$, and assume that $a\in G(\Sigma_Y)=G(\Sigma_{X^m})$.
    We will prove that $V_{X^m}(x_{i'},a')\cap S^m = (f_m)^{-1}(V_{X^{m+1}}(x_{i'},a')\cap S^{m+1})$.
    If \(a\neq a'\), then the loci \(V_{X^{m+1}}(x_{i'},a') \cap S^{m+1}\) and \(V_{X^{m+1}}(x_i,x_j,a) \cap S^{m+1}\) do not intersect.
    If \(a=a'\), then by
    observation (3) made in \cref{setup:loci}, we have \(i' \neq i,j\).
    If \(i'\neq 2\), then \(V_{X^{m+1}}(x_{i'},a') \cap S^{m+1}\) again does not touch the flipped locus.
    If \(i'=2\), then
    ${f_m}_{|S^m}\colon S^m\to S^{m+1}$ is an isomorphism.
    In any case, 
    we proved the claim, and then
    the desired description of the intersection $V_{X^m}(x_{i'},a')\cap S^m$ follows from the induction hypothesis. 
    \item For $0\leq q<L$, let $f_q:X^q\to X^{q+1}$ be the blowdown associated to a relevant primitive relation of the form $x_{i'}+a'=b'$, and assume that $a\notin G(\Sigma_Y)$. Then $(X,S_x)$ is an exceptional case, hence $q=0$, $L=2$, and the blowdowns $f_0 \colon X^0 \to X^1$ and $f_1 \colon X^1 \to X^L$ are associated to relevant primitive relations of the form $x_{i'}+a'=b'$ and $x_{j'}+c'=a'$, respectively. By case (II) above, the intersection $V_{X^m}(x_{j'},c')\cap S^m$ is the union of finitely many smooth fibers, respectively smooth points of fibers, of $\pi_m \colon S_m \to C$, one over each point of the intersection $C \cap V_W(c')$ if $j'=2$, respectively $j' \neq 2$. The desired description of the intersection $V_{X^m}(x_{i'},x_{j'},c')\cap S^m$ follows by intersecting $V_{X^m}(x_{j'},c')\cap S^m$ with $V_{X^m}(x_{i'})$.
\end{enumerate}

Now consider the case $0\leq m<L$, when $f_m\colon X^m\to X^{m+1}$ is the blowdown associated to a relevant primitive relation of the form $x_{i}+a=b$. It contracts $V_{X^m}(b)\subset X^m$ to $V_{X^{m+1}}(x_i, a)\subset X^{m+1}$.
If $i=2$, then ${f_m}_{|S^m}\colon S^m\to S^{m+1}$ is an isomorphism. Otherwise, ${f_m}_{|S^m}\colon S^m\to S^{m+1}$ is the blowup of $S^{m+1}$ at the finitely many points in $V_{X^{m+1}}(x_i,a)\cap S^{m+1}$. In any case, the composed map $\pi_m\colon S^m\to C$ is a morphism whose general fiber is isomorphic to $\PP^1$. 
For $0\leq q<m$, $f_q\colon X^q\to X^{q+1}$ is the blowdown associated to a relevant primitive relation of the form $x_{i'}+a'=b'$, with $i\neq i'$ and $a\neq a'$. If $a'\in G(\Sigma_Y)$, then, as in case (II), $V_{X^m}(x_{i'},a')\cap S^m = (f_m)^{-1}(V_{X^m}(x_{i'},a')\cap S^{m+1})$, and the desired description of the intersection $V_{X^m}(x_{i'},a')\cap S^m$ follows from the induction hypothesis. If instead $a'\notin G(\Sigma_Y)$, then $(X,S_x)$ is an exceptional case, hence $q=0$, $L=2$, $m=1$, and $a'=b$. There are three possible cases. For $i=2$, ${f_1}_{|S^1}$ is an isomorphism, hence the desired description of the intersection $V_{X^1}(x_{i'},b) \cap S^1= (f_1)^{-1}(V_{X^2}(x_{i'},x_i,a) \cap S^2)$ follows from case (II).
For $i'=2$, ${f_1}_{|S^1}$ is the blowup of $S^{2}$ at the finitely many points in $V_{X^{2}}(x_i,a)\cap S^{2}$, hence the intersection $V_{X^1}(x_{i'},b) \cap S^1= \Exc({f_1}_{|S^1})$ is the union of finitely many smooth curves of numerical class $x_i+a=b$ by \cref{rem:class contracted curve}, one over each point in $V_{X^{2}}(x_i,a)\cap S^{2}$, hence one over each point of the intersection $C \cap V_W(a)$ by (II).
Finally, for $\{i,i'\}=\{0,1\}$, the intersection $V_{X^1}(x_{i'},b) \cap S^1$ is empty as its image via ${f_1}_{|S^1}$ is $V_{X^2}(x_{i'},x_i,a) \cap S^2 = \varnothing$.
\end{proof}

Our next goal is to compare the intersection numbers $\ch_2(X^\ell) \cdot S^\ell$ and $\ch_2(X^{\ell-1}) \cdot S^{\ell-1}$.
For this purpose, we will make use of the following lemma. It is stated and proved in \cite{dJS06} for projective varieties \(X\) and \(Y\), but their proof extends verbatim to the proper case.

\begin{lemma}(\cite[Lemma 5.1]{dJS06})
\label{Lem:ch_blowup}
Consider the blowup diagram
\begin{center}
\begin{tikzcd}
E \arrow [r, hook, "j"] \arrow[d, "\sigma\coloneqq f_{|E}"] & X \coloneqq \text{Bl}_Z Y \arrow[d, "f"]\\
Z \arrow[r, hook] & Y
\end{tikzcd}
\end{center}
where both $Y$ and $Z$ are smooth proper varieties and $\codim_Y Z = c\geq 2$. 
Then we have the following relation between the 2nd Chern characters of $X$ and $Y$:
$$ \ch_2 (X) =
f^* \ch_2 (Y) + \frac{c+1}{2} E^2 - j_*\sigma^* c_1(\mathcal{N}_{Z/Y})
.$$
\end{lemma}

\begin{proposition}\label{prop:computation_blowup}
In the setting of Lemma \ref{Lem:ch_blowup},
suppose that \begin{tikzcd}
    X^{\ell-1}\arrow[r,"f_{\ell-1}"] &X^\ell 
\end{tikzcd} is the blowdown associated to the relevant primitive relation $x_i+a=b$, that is, $X^{\ell-1}\coloneqq\mathrm{Bl}_{Z}X^\ell$ where $Z\coloneqq V_{X^\ell}(x_i,a) \subset X^\ell$, and assume that $a\in G(\Sigma_Y)$.
Then 
\[
\ch_2(X^{\ell-1}) \cdot S^{\ell-1} \ = \
   \ch_2(X^\ell) \cdot S^\ell -
   \begin{cases}
       m & \text{if }i=2\\
       \frac{3}{2}m  &  \text{if }i\neq 2 \ ,
   \end{cases}
\]
where $m$ denotes the number of points of the intersection $C \cap V_W(a)$.
\end{proposition}

\begin{proof}
By \cref{Lem:ch_blowup}, 
\begin{align*}
    \ch_2(X^{\ell-1}) \cdot S^{\ell-1} 
    & = f_{\ell-1}^* \ch_2 (X^\ell) \cdot S^{\ell-1} + \frac{c+1}{2} E^2 \cdot S^{\ell-1}  - j_*\sigma^* c_1(\mathcal{N}_{Z/{X^\ell}}) \cdot S^{\ell-1}  \\
    & = \ch_2 (X^\ell) \cdot  S^\ell + \tfrac{3}{2} \,E^2 \cdot S^{\ell-1} - c_1(\mathcal{N}_{Z/{X^\ell}})\cdot (S^{\ell} \cap Z).
\end{align*} 

If $i=2$, then, by \cref{lemma:intersections}, $S^\ell \cap Z$ is the union of finitely many smooth fibers of $\pi_\ell\colon S^\ell\to C$, one over each of the $m$ points in the intersection $C\cap V_W(a)$. Each of these fibers has class $x_0+x_1+x_2=0$ and self-intersection $0$. It follows that $(f^{\ell-1})_{|S^{\ell-1}}$ is an isomorphism, hence $E^2 \cdot S^{\ell-1}=0$, and  $$c_1(\mathcal{N}_{Z/{X^\ell}})\cdot (S^{\ell} \cap Z)=\mathrm{deg}((\mathcal{O}_{X^\ell}(V(x_2)) \oplus \mathcal{O}_{X^\ell}(V(a))|_{S^\ell \cap Z})=m.$$ 
We conclude that $\ch_2(X^{\ell-1}) \cdot S^{\ell-1} =\ch_2(X^\ell) \cdot S^\ell -m$.

Suppose now that $i \neq 2$. By \cref{lemma:intersections}, $S^\ell \cap Z$ is
the union of finitely many smooth points of fibers of $\pi_\ell\colon S^\ell\to C$, one over each of the $m$ points in the intersection $C\cap V_W(a)$.  It follows that $(f^{\ell-1})_{|S^{\ell-1}}: S^{\ell-1} \to S^\ell$ is the blowup at the $m$ points of $S^\ell \cap Z$. Hence, $E^2 \cdot S^{\ell-1}=-m$ and $c_1(\mathcal{N}_{Z/{X^\ell}})\cdot (S^{\ell} \cap Z)=0$. This shows that $\ch_2(X^{\ell-1}) \cdot S^{\ell-1} = \ch_2(X^\ell) \cdot S^\ell - \frac{3}{2}m$.
\end{proof}

\begin{proposition}\label{prop:computation_blowup_exc}
Suppose that $(X,S_x)$ is an exceptional case, hence $L=2$. Consider the blowdowns $f_0 \colon X^0 \to X^1$ and $f_1 \colon X^1 \to X^L$ associated to the relevant primitive relations $x_i+a=b$ and $x_j+c=a$, respectively.
Then 
\[
\ch_2(X^{0}) \cdot S^{0} \ = \
   \ch_2(X^L) \cdot S^L -
   \begin{cases}
       \frac{5}{2}m & \text{if }j=2, \\
       2m  &  \text{if }i=2, \\
       \frac{3}{2}m  & \text{otherwise},
   \end{cases}
\]
where $m$ denotes the number of points in the intersection $C \cap V_W(c)$.
\end{proposition}

\begin{proof}
By \cref{prop:computation_blowup}, 
\[
\ch_2(X^{1}) \cdot S^{1} \ = \
   \ch_2(X^L) \cdot S^L -
   \begin{cases}
       m & \text{if }j=2,\\
       \frac{3}{2}m  &  \text{if }j\neq 2 .
   \end{cases}
\]
In the setting of Lemma \ref{Lem:ch_blowup}, set $Z \coloneqq V_{X^1}(x_i,a)$. 
If $\{i,j\}=\{0,1\}$, then $S^1$ is disjoint from $Z$ by \cref{lemma:intersections}, and by \cref{Lem:ch_blowup} we have the equality 
\[
\ch_2(X^{0}) \cdot S^{0} = \ch_2(X^1) \cdot S^1 = \ch_2(X^L) \cdot S^L - \tfrac{3}{2}m.
\]

If $i=2$ then, by \cref{lemma:intersections}, $S^1 \cap Z$ is the union of $m$ smooth curves of numerical class $x_j+c=a$, contained in fibers of $\pi_1\colon S^1\to C$, one over each point of the intersection $C\cap V_W(c)$. 
It follows that $(f_{0})_{|S^{0}}$ is an isomorphism
and \(E\cap S^0 \cong S^1\) is a union of \(m\) rational curves,
hence $E^2 \cdot S^{0}=-m$, and $$c_1(\mathcal{N}_{Z/{X^1}})\cdot (S^{1} \cap Z)=\mathrm{deg}\Big(\big(\mathcal{O}_{X^1}(V(x_2)) \oplus \mathcal{O}_{X^1}(V(a))\big)\big|_{S^1 \cap Z}\Big)=-m.$$ 
By \cref{Lem:ch_blowup}, we conclude that 
\[\ch_2(X^{0}) \cdot S^{0} =\ch_2(X^1) \cdot S^1-\frac{3}{2}m +m =\ch_2(X^L) \cdot S^L -2m.
\]

If $j=2$ then, by \cref{lemma:intersections}, $S^1 \cap Z$ is the union of $m$ smooth points of fibers of $\pi_1\colon S^1\to C$, one over each point of the intersection $C\cap V_W(c)$.
It follows that $(f_{0})_{|S^{0}}$ is the blowup at the $m$ points of $S^1 \cap Z$. Hence, $E^2 \cdot S^{0}=-m$ and $c_1(\mathcal{N}_{Z/{X^1}})\cdot (S^{1} \cap Z)=0$. This shows that 
\[\ch_2(X^{0}) \cdot S^{0} =\ch_2(X^1) \cdot S^1 - \tfrac{3}{2}m =\ch_2(X^L) \cdot S^L -\tfrac{5}{2}m.
\qedhere
\]
\end{proof}

\begin{proposition}\label{prop:computation flip}
 Let $f_{\ell-1}:X^{\ell-1} \dashrightarrow X^\ell$ be the flip corresponding to the relevant primitive relation $x_i+x_j+a=b+c$. Then
$$
\ch_2(X^{\ell-1}) \cdot S^{\ell-1} = \ch_2(X^\ell) \cdot S^\ell -
\begin{cases}
       \frac{5}{2}m & \text{if }i=2 \text{ or }j=2\\
       0  &  \text{otherwise, }
   \end{cases}
$$
where $m$ denotes the number of points of the intersection $C \cap V_W(a)$.
\end{proposition}

\begin{proof} The map $f_{\ell-1}$ is the composition
\[\begin{tikzcd}
	& \tilde{X^\ell} \\
	{X^{\ell-1}} && X^\ell
	\arrow["{\alpha:\, b+c=z}"'{pos=0.5}, from=1-2, to=2-1]
	\arrow["{\beta:\, x_i+x_j+a=z}"{pos=0.5}, from=1-2, to=2-3]
	\arrow["{f_{\ell-1}}", dashed, from=2-1, to=2-3]
\end{tikzcd}\]
of the blowup $\alpha$ of $V_{X^{\ell-1}}(b,c)$ with exceptional divisor $\bar{E}\coloneqq V(z)$, and the blowdown $\beta$, which contracts $\bar{E}$ to $Z:=V_{X^\ell}(x_i,x_j,a)$.
We denote by $\tilde{S}$ the strict transform of $S^\ell$ in $\tilde{X^\ell}$.

If $\{i,j\}=\{0,1\}$, then $S^\ell$ is disjoint from $Z$ by \cref{lemma:intersections}, $f_{\ell-1}:X^{\ell-1} \dashrightarrow X^\ell$ is an isomorphism in a neighborhood of $S^{\ell-1}$, and hence $\ch_2(X^{\ell-1}) \cdot S^{\ell-1} = \ch_2(X^\ell) \cdot S^\ell$. 

Assume that $2\in \{i,j\}$. By \cref{lemma:intersections}, $S^\ell \cap Z$ is the union of finitely many smooth points of fibers of $\pi_\ell\colon S^\ell\to C$, one over each point of the intersection $C\cap V_W(a)$.
Then $\tilde{S}$ is isomorphic to the blowup of the surface $S^\ell$ at the $m$ points of $S^\ell \cap Z$.
It follows that $\bar{E} \cap \tilde{S}$ is a union of $m$ curves, each of self-intersection $-1$.
By \cref{Lem:ch_blowup}, 
\begin{align*}
    \ch_2(\tilde{X^\ell}) \cdot \tilde{S} 
    & = \beta^* \ch_2 (X^\ell) \cdot \tilde{S} + \frac{c+1}{2} \bar{E}^2 \cdot \tilde{S}  - j_*\sigma^* c_1(\mathcal{N}_{Z/{X^\ell}}) \cdot \tilde{S}  \\
    & =  \ch_2 (X^\ell) \cdot  S^\ell + 2 \,\bar{E}^2 \cdot \tilde{S} -\cancel{c_1(\mathcal{N}_{Z/{X^\ell}})\cdot (S^\ell \cap Z)}\\
    & =  \ch_2 (X^\ell) \cdot  S^\ell -2m.
\end{align*}

Next, consider the contraction $\alpha \colon \tilde{X^\ell} \to X^{\ell-1}$ associated to the class $b+c=z$, which maps $\bar{E}$ to $V_{X^{\ell-1}}(b,c)$. Here $Z':=V_{X^{\ell-1}}(b,c) \subset X^{\ell-1}$ and hence $\mathrm{codim}_{X^{\ell-1}}(Z')=2$.
Moreover, by \cref{rem:class contracted curve}, $\bar{E} \cap \tilde{S}$ consists of $m$ curves of class $x_i+x_j+a=z$, which are not contracted by $\alpha$. Thus $\alpha$ maps $\tilde{S}$ isomorphically onto $S^{\ell-1}$.
Therefore, the intersection $Z' \cap S^{\ell-1}=\alpha(\bar{E} \cap \tilde{S})$ consists of $m$ curves of class $x_i+x_j+a=b+c$,
since $V(v)=\alpha^*V(v)$ for $v\in\{x_i, x_j, a\}$ and $V(w)=\alpha^*V(w)-V(z)$ for $w\in\{b,c\}$.
So:
$$c_1(\mathcal{N}_{Z'/{X^{\ell-1}}})\cdot (S^{\ell-1} \cap Z')=\mathrm{deg}\Big(\big(\mathcal{O}_{X^{\ell-1}}(V(b)) \oplus \mathcal{O}_{X^{\ell-1}}(V(c))\big)\big|_{S^{\ell-1} \cap Z'}\Big)=-2m.$$ 
By Lemma \ref{Lem:ch_blowup}, 
    \begin{align*}
    \ch_2(\tilde{X^\ell}) \cdot \tilde{S} 
    & = \alpha^* \ch_2 (X^{\ell-1}) \cdot \tilde{S} + \frac{c+1}{2} \bar{E}^2 \cdot \tilde{S}  - j_*\sigma^* c_1(\mathcal{N}_{Z'/X^{\ell-1}}) \cdot \tilde{S}  \\
    & =  \ch_2 (X^{\ell-1}) \cdot  S^{\ell-1} + \tfrac{3}{2} \,\bar{E}^2 \cdot \tilde{S} -c_1(\mathcal{N}_{Z'/X^{\ell-1}}) \cdot (S^{\ell-1} \cap Z')\\
    & =  \ch_2 (X^{\ell-1}) \cdot  S^{\ell-1} -\tfrac{3m}{2}-(-2m)=\ch_2 (X^{\ell-1}) \cdot  S^{\ell-1}+\tfrac{m}{2}.
\end{align*}
Combining the above, we get
\[
    \ch_2 (X^{\ell-1}) \cdot  S^{\ell-1}
    =\ch_2(\tilde{X^\ell}) \cdot \tilde{S}-\tfrac{m}{2} =\ch_2(X^\ell) \cdot S^\ell -\tfrac{5}{2}m.
    \qedhere
\]
\end{proof}

\begin{corollary}
    Let the notation and assumptions be as in \cref{setup:loci}. Then the surface $S^0$ satisfies the inequality $\ch_2(X)\cdot S^0 \leq \ch_2(Y)\cdot S \leq 0$, and in particular $X$ is not $2$-Fano.
\end{corollary}
\begin{proof}
    If $(X,S_x)$ is not exceptional, apply \cref{prop:computation_blowup} to each blowdown $f_\ell$ for $0 \leq \ell < L$, and \cref{prop:computation flip} to each flip $f_\ell$ for $L \leq \ell < k$. If $(X,S_x)$ is exceptional, apply \cref{prop:computation_blowup_exc}, and then \cref{prop:computation flip} to each flip $f_\ell$ for $L \leq \ell < k$. In both cases, we obtain the desired inequality. 
\end{proof}

% ---
\section{Appendix: Towards the case \texorpdfstring{\(m(X) = 3\)}{m(X)=3}}
\label{section: complications for m=3}
% ---

One may wonder if the program outlined at the beginning of \cref{section:strategy} may work just as well for \(m(X) \geq 3\).
In this appendix, we discuss the next case \(m(X) = 3\), indicating which parts of the proof of Step 1 extend to this setting, providing examples, highlighting the new complications that arise, and outlining possible approaches to address them. As we shall see, a significantly more detailed case-by-case analysis is required.

The new difficulties that arise in the case \(m(X) = 3\) can be summarized as follows. 
\begin{itemize}
\item There are more types of relevant primitive relations (\cref{table: RPR for m=3}), several of which are not manifestly contractible. 
\item There are more exceptional cases to be treated separately (\cref{more_exceptional cases}).
\item The birational transformations associated to relevant primitive relations may lead to singular varieties (\cref{ex:singular_blowup} and \cref{ex:singular_flip}), with the theory of primitive relations being more involved in the singular setting \cite{CoxvonRenesse2009}.
\item In order to avoid singular varieties, if at all possible, it may be necessary to contract relations that are not relevant.
The existence of such auxiliary relations and the behavior of relevant primitive relations with respect to the associated contractions should be investigated.  (See \cref{ex:singular_blowup} and \cref{ex:singular_flip}.)
\end{itemize}
In what follows, we discuss these difficulties through examples and summarize partial progress in the case \(m(X) = 3\) without going into too much detail.
Throughout this appendix, let \(X\) be a toric Fano manifold with \(m(X) = 3\) and a fixed centered primitive relation \(x_0+x_1+x_2+x_3 = 0\).

As before, we start by classifying all possible relevant primitive relations.
They are displayed in \cref{table: RPR for m=3}, together with some of their properties.

\begin{table}[h]
\begin{tabular}{|c|l|l|l|}
\hline
\textbf{Type} & \begin{tabular}[c]{@{}l@{}} \textbf{Relevant primitive} \\ \textbf{relation} \end{tabular} & \textbf{Provable properties} & \begin{tabular}[c]{@{}l@{}} \textbf{Appears in} \\ \textbf{Macaulay2 ID} \end{tabular} \\ \hline
1    & \( x_i + a = b \)                     & Extremal; \(b \neq x_j\)                                                                    & (6, 539)                     \\ \hline
2    & \( x_i + x_j + a = b + c \)           & Extremal                                                                                                       & \begin{tabular}[c]{@{}l@{}} (6, 2266), (6, 2268): \\ with \(c=x_k\) \end{tabular}     \\ \hline
3    & \( x_i + x_j + a = 2b \)              & \begin{tabular}[c]{@{}l@{}} Extremal; \(b \neq x_k\); \\ appears with an auxiliary \\ extremal PR \(b+c = a\)    \end{tabular}                                                              & \begin{tabular}[c]{@{}l@{}} (6, 2170), (6, 2264), \\(6, 2272) \end{tabular}              \\ \hline
4    & \( x_i + x_j + a = b \)               & \begin{tabular}[c]{@{}l@{}} \(b\neq x_k\); \\ appears with either \\ (1) RPR \(x_k+x_\ell+b=a\), or \\ (2) RPRs \(x_k+b=c\), \(x_\ell+c=a\) \end{tabular} & \begin{tabular}[c]{@{}l@{}} (6, 2170), (6, 2264), \\(6, 2272) \end{tabular}              \\ \hline
5    & \( x_i + x_j + x_k + a = b + c + d \) & Extremal; \(b,c,d \neq x_\ell\)                                                                                                       & (6, 772), (6, 4347)                     \\ \hline
6    & \( x_i + x_j + x_k + a = 2b + c \)    & \begin{tabular}[c]{@{}l@{}} Extremal; \(b,c \neq x_\ell\)  \end{tabular}                                                                                & (6, 2268)                   \\ \hline
7    & \( x_i + x_j + x_k + a = 3b \)        & Cannot occur                                                                                                 & --                      \\ \hline
8    & \( x_i + x_j + x_k + a = b + c \)     & \begin{tabular}[c]{@{}l@{}} \(b,c\neq x_\ell\); \\ if there are no type 1,\\ appears with PR \(x_\ell + b + c = a\)\end{tabular} & (6, 2268)                    \\ \hline
9    & \( x_i + x_j + x_k + a = 2b \)        & Cannot occur                                                                                                 & --                      \\ \hline
10   & \( x_i + x_j + x_k + a = b \)         & \begin{tabular}[c]{@{}l@{}} \(b\neq x_\ell\); \\appears with \(x_\ell+b=a\) \end{tabular} & (6, 539), (6, 4348)                     \\ \hline
\end{tabular}
\caption{RPR for toric Fano manifolds with \(m=3\).}
\label{table: RPR for m=3}
\end{table}

One may proceed by first contracting relevant primitive relations of type 1.
The composition of such blowdowns preserves the fixed centered primitive relation and does not create new opponents by \cref{lem:P and opponents after blowdown}.
We can still show that, after contracting type 1 relations, only special kinds of relations may appear.
Cases in which new relevant primitive relations appear are referred to as exceptional, and may need to be treated separately, as in the case $m(X)=2$. One can show that exceptional cases occur exactly when $X$ has one of the following sets of primitive relations:
\begin{align} \label{more_exceptional cases} 
\begin{split}
       & \bullet \ x_1+a=b, \quad x_2+b=c, \quad x_3+c=d \quad \text{ and } \quad x_0+d=a \, ; \\
       & \bullet \ x_1+x_2+a=b, \quad x_3+b=c \quad \text{ and } \quad x_0+c=a \, ; \\ 
       & \bullet \ x_1+x_2+a=b+c \quad \text{ and } \quad x_0+d=a \, .
\end{split}
\end{align}

After eliminating relations of type 1, the situation is not as clear as in the case \(m(X) = 2\).
The following example shows that we need to either abandon the realm of smooth varieties when performing the birational transformations, or resort to contracting relations that are not relevant.

\begin{example}\label{ex:singular_blowup}
    Let \(X\) be a toric variety with Macaulay2 ID (6, 2170) or (6, 2264).
    It is a Fano 6-fold of Picard rank 4 with the following 9 primitive relations:
    \begin{align*}
        s &: \ x_0 + x_1 + x_2 + x_3 = 0,
        \\ p_1 &: \ x_0 + x_1 + t = 2a,
        \\ r_1 &: \ x_0 + x_1 + b = a,
        \\ p_2 &: \ x_2 + x_3 + t = 2b,
        \\ r_2 &: \ x_2 + x_3 + a = b,
        \\ p &: \ a+b=t,
    \end{align*}
    \[
        \begin{array}{ll|c|c}
             && \text{ID (6,~2170)} & \text{ID (6,~2264)} \\
             \hline
             q_1 :& u_0+u_1+u_2+a = & x_0 + x_1 + x_2 & 2x_0 + x_1, \\
             q_2 :& u_0+u_1+u_2+t = & x_2 + a & x_0+a, \\
             q_3 :& u_0+u_1+u_2+b = & x_2 & x_0.
        \end{array}
    \]

    The manifold \(X\) has 4 extremal primitive relations: \(p_1, p_2, p, q_1\).
    The remaining primitive relations can be expressed from the extremal relations as follows:
    \begin{align*}
        s &= 2p + p_1 + p_2, \\
        r_1 &= p+p_1, \\
        r_2 &= p+p_2, \\
        q_2 &= p_1+q_1, \\
        q_3 &= p+p_1+q_1.
    \end{align*}

    The following are all the relevant primitive relations: \(p_1,p_2,r_1,r_2\).
    The \(r_i\)'s are not contractible by \cref{prop: contractible}: \(\langle a,t \rangle\) and \(\langle b,t \rangle\) are cones in \(\Sigma_X\), but \(\{ x_0,x_1,t \}\) and \(\{ x_2,x_3,t \}\) are primitive collections.
    
    If we attempt to contract only relevant primitive relations, then we are naturally forced to contract one of the \(p_i\)'s, but then the result is a singular variety.
    We immediately run into the following problem: 
    Sato's description of how primitive collections change under blowup / blowdown is currently available only in the smooth setting.    

    However, in this example, we can contract \(p : \ a+b=t\), which gives us a smooth toric Fano manifold \(X^1\) with primitive relations \(s, r_1, r_2, q_1, q_3\), out of which \(r_1, r_2, q_1\) are extremal, and \(s = r_1 + r_2\), \(q_3 = q_1 + r_1\).
    We can then contract \(r_1\), yielding \(X^1 \to X^2\), and the latter is a \(\PP^3\)-bundle (globally, not just on a big open subset).
\end{example}

In fact, it can be shown that whenever $X$ has a primitive relation of type 3, say 
\[x_0 + x_1 + t = 2a,\] 
it also has an auxiliary extremal primitive relation of the form \(\alpha : \ a+b = t\).
One can therefore contract \(\alpha\) to eliminate the corresponding relevant primitive collection \(\{x_0,x_1,t\}\).
In this case, it is necessary to analyze whether new relevant primitive relations arise.

A similar approach can be used to treat primitive relations of type 6, when not all coefficients on the right-hand side are equal to 1.
We illustrate this situation in the example that follows, where we choose to contract another auxiliary primitive relation.

\begin{example}\label{ex:singular_flip}
    Let \(X\) be the Fano toric manifold with Macaulay2 ID (6, 2268).
    It is 6-dimensional of Picard rank 4.
    We use the following notation for the set of its generators and primitive relations:
    \begin{align*}
        & G(\Sigma_X) = \{x_0,x_1,x_2,x_3,y_0,y_1,y_2,y_3,a,b\}; \\
        & \begin{array}{ccl cc ccl}
            s_x &:& x_0+x_1+x_2+x_3=0,
            &&& s_y &:& y_0+y_1+y_2+y_3 = 0, \\
            r_6 &:& x_0+x_1+x_2+a = 2y_0+y_1,
            &&& p &:& x_3+y_1+a = 2b, \\
            r_8 &:& x_0+x_1+x_2+b = y_0+y_1,
            &&& p' &:& y_2+y_3+a=y_0+x_3, \\
            \alpha &:& y_0+b=a,
            &&& q &:& y_0+y_1+x_3 = b, \\
            && 
            &&& q' &:& y_2+y_3+b = x_3.
        \end{array}
    \end{align*}
    Relations \(p,p',r_6,\alpha\) are extremal, and the rest can be expressed from these as follows:
    \begin{align*}
        \begin{array}{l cc l cc l}
             s_x = r_6+p+2\alpha,   &&& r_8 = r_6+\alpha, \\
             s_y = p+p'+2\alpha,    &&& q = p+\alpha,  &&& q' = p'+\alpha.
        \end{array}
    \end{align*}
    
    If we choose the centered primitive relation \(s_x\) to perform our algorithm, then \(r_6\) and \(r_8\) are all the relevant primitive relations, which are of types 6 and 8, respectively.
    We notice that the right hand side of \(r_6\) has a coefficient of \(2\), hence performing a flip with respect to this relation will result in a singular variety.
    At the same time, \(r_8\) is not extremal, and in fact also not contractible by \cref{prop: contractible}, because \(\{y_0,b\}\) does not span a cone.

    In this case, we can start by contracting the auxiliary relation \(\alpha\), obtaining a blowdown \(X \to X^1\) with primitive relations
    \(\PR(X^1) = \{s_x, r_8, s_y, q, q'\}\).
    Now \(r_6\) disappears and \(r_8\) becomes extremal, so we next flip with respect to \(r_8\) to get \(X^1 \dashrightarrow X^2\).
    The primitive relations on the resulting manifold are \(\PR(X^2) = \{s_x, r_8^{\text{flip}}, q'\}\), where \(r_8^{\text{flip}} : y_0+y_1 = x_0+x_1+x_2+b\).
    Now, \(s_x\) is contractible on \(X^2\), hence it is a global \(\PP^3\)-bundle.
\end{example}

Contracting auxiliary primitive relations may help avoid dealing with singular varieties. Their existence must be investigated, as well as the possible emergence of new relevant primitive relations after contraction. New ideas may be required to treat other types of relevant primitive relations.

\printbibliography
\Addresses
\end{document}